\documentclass{article}
\usepackage[a4paper,left=2.5cm,right=2.5cm, top=3cm, bottom=3cm]{geometry}

\usepackage{graphicx}
\usepackage{multirow}
\usepackage{amsmath,amssymb,mathrsfs,bbm,listings}
\usepackage{mathtools}
\usepackage{amsthm}
\usepackage[dvipsnames]{xcolor}

\usepackage{tabularray}

\newtheorem{theorem}{Theorem}[section]
\newtheorem{assumption}[theorem]{Assumption}
\newtheorem{lemma}[theorem]{Lemma}
\newtheorem{proposition}[theorem]{Proposition}
\newtheorem{corollary}[theorem]{Corollary}
\newtheorem{remark}[theorem]{Remark}

\numberwithin{equation}{section}

\usepackage{array}
\usepackage{enumitem}

\newcommand{\calA}{{\mathcal A}}
\newcommand{\calF}{{\mathcal F}}
\newcommand{\calL}{{\mathcal L}}
\newcommand{\calH}{{\mathcal H}}
\newcommand{\calM}{{\mathcal M}}
\newcommand{\calT}{{\mathcal T}}
\newcommand{\IN}{\mathbb{N}}
\newcommand{\IP}{\mathbb{P}}
\newcommand{\IR}{\mathbb{R}}
\newcommand{\deO}{{\partial\Omega}}
\newcommand{\uu}[1]{\hbox{\boldmath$#1$}}   
\newcommand{\bn}{{\uu n}}  
\newcommand{\bm}{{\uu m}}
\newcommand{\bt}{{\uu t}}
\newcommand{\bv}{{\uu v}}
\newcommand{\bw}{{\uu w}}
\newcommand{\bx}{{\uu x}}
\newcommand{\by}{{\uu y}}
\newcommand*{\Norm}[1]{\left\|#1\right\|}
\newcommand*{\N}[1]{\left\|#1\right\|}
\newcommand*{\abs}[1]{\left|#1\right|}
\newcommand*{\jmp}[1]{[\![#1]\!]}
\newcommand*{\mvl}[1]{\{\!\!\{#1\}\!\!\}}
\newcommand{\iin}{\;\text{in}\;}
\newcommand{\oon}{\;\text{on}\;}
\newcommand{\Th}{{(\calT_h)}}
\newcommand{\Fh}{\calF_h}
\newcommand{\sign}{{\mathrm{sign}}}
\newcommand{\half}{\frac{1}{2}}

\newcommand\Pp[2]{\mathbb{P}^{#1}(#2)}

\newcommand{\FhI}{\calF_h^{\mathrm{I}}}
\newcommand{\Fhtwo}{\calF_h^2}

\newcommand{\Fhzero}{\calF_h^0}
\newcommand{\FhD}{\calF_h^{\mathrm{D}}}
\newcommand{\Ctr}{C_{\mathrm{tr}}}

\newcommand{\Capp}{C_{\mathrm{app}}}
\newcommand{\qT}{{\mathbb{Q\!T}}}
\newcommand{\W}{\mathbf{W}}
\newcommand{\Tc}{\calT^{\#}}
\newcommand{\eT}{{\mathbb{E\!T}}}
\newcommand{\mi}{{\boldsymbol{i}}}

\allowdisplaybreaks[4]

\usepackage{stmaryrd}
\SetSymbolFont{stmry}{bold}{U}{stmry}{m}{n}

\newcommand{\vertiii}[1]{{\left\vert\kern-0.25ex\left\vert\kern-0.25ex\left\vert #1 
		\right\vert\kern-0.25ex\right\vert\kern-0.25ex\right\vert}}

\usepackage{tikz}
\usetikzlibrary{
	calc,
	positioning,
	arrows.meta,
	decorations.markings,
	matrix,
	spy,
	shapes.geometric
}

\usepackage{pgfplots}
\usepgfplotslibrary{groupplots}
\pgfplotsset{compat=1.18}

\usepackage{csvsimple} 
\usepackage{booktabs}

\usepackage{pgfplotstable} 
\usepackage{colortbl}
\makeatletter

\usepackage[hidelinks]{hyperref}
\definecolor{myblue}{rgb}{0,0,0.6}   
\hypersetup{
	colorlinks,
	linktoc=section,
	citecolor=red,
	linkcolor=myblue,
	urlcolor=myblue,
	pdfborder={0 0 0}
}
\usepackage{cleveref} 

\begin{document}
	\title{A discontinuous Galerkin method for elliptic--hyperbolic equations: the Tricomi problem}
	\author{Chiara Perinati\thanks{Department of Mathematics, University of Pavia, Italy (\href{mailto:chiara.perinati01@universitadipavia.it}{chiara.perinati01@universitadipavia.it})} , 
	Lise-Marie Imbert-G\'erard\thanks{Department of Mathematics, University of Arizona, USA  (\href{mailto:lmig@arizona.edu}{lmig@arizona.edu})} ,
	Andrea Moiola\thanks{Department of Mathematics, University of Pavia, Italy (\href{mailto:andrea.moiola@unipv.it}{andrea.moiola@unipv.it})} ,
	Paul Stocker\thanks{Faculty of Mathematics, University of Vienna, Austria (\href{mailto:paul.stocker@univie.ac.at}{paul.stocker@univie.ac.at})}}
	\date{\today}
	\maketitle	
\begin{abstract}\noindent
    We present and analyze a discontinuous Galerkin method for the numerical solution of a class of second-order linear mixed-type partial differential equations, i.e.\ equations that change their nature from elliptic to hyperbolic through the computational domain. 
	Well-posedness of the discrete problem is established via coercivity in an energy norm, achieved through the Morawetz multiplier technique.
	We derive $hp$-\emph{a priori} error estimates in the energy norm, which we use to prove convergence rates for standard and quasi-Trefftz polynomial spaces. Numerical experiments validate the theoretical results.
\end{abstract}

\paragraph{Keywords.} 
Discontinuous Galerkin method; mixed-type equations; Tricomi problem; Morawetz multipliers; $hp$-convergence.
	
\paragraph{Mathematics Subject Classification (MSC2020).}  65N15, 65N30, 35M12, 41A10, 41A25.

\section{Introduction} \label{s:introduction}

\paragraph{Model Problem.}
We consider a class of second-order linear partial differential equations (PDEs) of mixed type. 
In particular, we focus on the Frankl operator:
\begin{equation}\label{eq:PDE}
	\calL u:= K u_{xx} + u_{yy},
\end{equation}
where the coefficient $K:\IR\to\IR$ depends only on the variable $y$, therefore we write $K=K(y)$. The function $K(y)$ changes sign in the computational domain $\Omega\subset \IR^2$, making the operator $\calL$ elliptic in a region and hyperbolic in another one.
A classical example is the Tricomi equation, which corresponds to $K(y)=y$,  whose type changes across the line $y=0$. In this work, we assume that the coefficient satisfies
\begin{equation}\label{eq:kappa}
K=K(y),\quad K(y)y>0 \ \text{ if }\ y\neq 0, \quad K\in C^0(\overline\Omega)\cap C^1(\Omega), \quad{K'>0},
\end{equation}
so that the operator~\eqref{eq:PDE} is elliptic where $y>0$ and hyperbolic when $y<0$.
The parabolic curve is $\{(x,y)\in \Omega \mid K(y)=0\}=\Omega\cap\{y=0\}\ne\emptyset$.

The boundary $\deO$ is decomposed into distinct parts $\Gamma_0$, $\Gamma_1$ and $\Gamma_2$.
We assume that the elliptic part $\Gamma_0:=\deO\cap\{y>0\}$ is a Lipschitz curve and that the hyperbolic boundary $\deO\cap\{y<0\}$ is the union of two characteristic curves $\Gamma_1$ and $\Gamma_2$.
Without loss of generality we fix the points where $\deO$ intersects the parabolic line at $\Gamma_1\cap\{y=0\}=(-1,0)$ and $\Gamma_2\cap\{y=0\}=(1,0)$, and consequently $\Gamma_1\cap\Gamma_2=(0,y_c)$ for some $y_c<0$.
The characteristic curves are given by
\begin{equation}\label{eq:chareq}
	x = -1 + \int_y^0 \sqrt{-K(t)}\, \mathrm{d} t \! \quad \text{on } \Gamma_1, \qquad 
	x = 1 - \int_y^0 \sqrt{-K(t)}\, \mathrm{d} t \! \quad \text{on } \Gamma_2,
\end{equation}
and satisfy the characteristic relation
\begin{equation}\label{eq:relationchar}
	K n_x^2+n_y^2=0,
\end{equation}
where $\bn=(n_x,n_y)^\top$ denotes the outward normal vector to the boundary $\partial \Omega$.
Figure~\ref{fig:domain} shows a sketch of~$\Omega$.

We consider the following boundary value problem for the Frankl equation:
\begin{alignat}{2}
	\calL u= K u_{xx} + u_{yy}=f & \quad \text{in  $\Omega\subset\IR^2$,}\label{eq:Frankl}\\
	u = g &\quad  \text{on $ \Gamma_0\cup \Gamma_1$,} \label{eq:bndc:Dir}
\end{alignat}
with source term $f\in L^2(\Omega)$ and sufficiently regular boundary datum $g$.

\begin{figure}
	\begin{center}
		\begin{tikzpicture}[scale=0.8]
			\begin{axis}[
				axis lines = middle,
				axis line style = {->},
				xlabel = {$x$},
				ylabel = {$y$},
				xlabel style = {right},
				ylabel style = {above},
				xtick = \empty,
				ytick = \empty,
				enlargelimits = true,
				domain=-1:1,
				samples=100,
				clip=false,
				width=8cm,
				height=8cm
				]
				\addplot [domain=-1:0, color=black, thick] {-(3/2*(x+1))^(2/3)};
				\addplot [domain=0:1, color=black, thick] {-(-3/2*(x-1))^(2/3)};
				\addplot[
				color=black,
				thick,
				smooth
				] coordinates {
					(1, 0)
					(0.7, 1.2)
					(0,1)	
					(-0.4,0.9)
					(-0.8, 1)
					(-1.1,0.9)
					(-1,0)
				};
				\node at (axis cs: 0.3,0.3) [anchor=south] {\large $\Omega$};
				\node at (axis cs: -1.05,0) [anchor=south,xshift=-5pt,yshift=-16pt] {\large $-1$};
				\node at (axis cs: 1.05,0) [anchor=south,xshift=2pt,yshift=-15pt] {\large $1$};
				\node at (axis cs: -0.13,-1.25) [anchor=south,xshift=2pt,yshift=-15pt] {\large $y_c$};
				\node at (axis cs: -0.5 ,-0.9) [anchor=east] {\large $\Gamma_1$};
				\node at (axis cs: 0.5,-0.9) [anchor=west] {\large $\Gamma_2$};
				\node at (axis cs: -0.4,1) [anchor=south] {\large $\Gamma_0$};
			\end{axis}
		\end{tikzpicture}
	\end{center}
	\caption{Domain $\Omega$.}
	\label{fig:domain}
\end{figure}
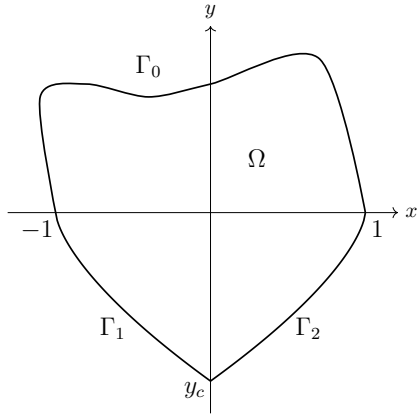
Equation~\eqref{eq:Frankl} is also known as the Chaplygin equation, and the boundary value problem~\eqref{eq:Frankl}–\eqref{eq:bndc:Dir}, where $\Gamma_1$ and $\Gamma_2$ are characteristic curves, is referred to as the \emph{Tricomi problem} (\cite[\S 1.2]{kim1990analytical}).

\paragraph{Motivations and applications.}
Partial differential equations of mixed type were first studied by Tricomi~\cite{tricomi1923sulle};  see~\cite{cibrario1955equazioni} for a classification of such equations.
They arise naturally in physical problems in which the type of the governing operator changes across the domain.
Numerous examples of applications of elliptic–hyperbolic equations can be found in~\cite[\S2.2]{otway2015elliptic}, ranging from pure mathematics to plasma physics, fluid and traffic flow, cosmology and car engineering.
In particular, we recall the application in the theory of transonic gas flows, where the change of operator type corresponds to the transition between subsonic and supersonic regimes.
In this context, the problem is typically formulated in the hodograph plane, where the unknown $u$ represents a stream function of the flow and the independent variables  $x$ and $y$ denote the flow angle and a scaled flow speed, respectively;
see, for example,
\cite[Chapters 1 and 5]{bers2016mathematical},
\cite[Chapter X, \S 1]{dautray1999mathematical},
\cite[\S 1.1]{kim1990analytical},
\cite[\S 4]{morawetz2004mixed},
\cite{morawetz1981lectures}
and
\cite[\S 1]{rassias1990lecture}.
Recent developments include spherical gravitational collapse~\cite{Ripley_2019,PhysRevD.99.084014} and rotating wave solutions of a nonlinear wave equation~\cite{Kubler}.

\paragraph{Well-posedness of the continuous problem.}
The Tricomi problem~\eqref{eq:Frankl}--\eqref{eq:bndc:Dir} is referred to as an \emph{open} problem, since Dirichlet boundary conditions are imposed only on a portion of the boundary, whereas problems with Dirichlet conditions on the entire boundary are called \emph{closed}.
Closed problems are generally ill posed in the class of strong solutions~\cite{morawetz2004mixed}.
For well-posedness results concerning closed problems, we refer to~\cite{lupo2007closed,payne2005weak,payne2007multiplier}. 
Open problems instead are well posed
under suitable assumptions, such as a star-shaped condition on the elliptic boundary $\Gamma_0$ or a Frankl-type condition on $K$.
Existence and uniqueness theorems for the open Tricomi problem have been obtained, e.g., in~\cite{aziz1978uniqueness,friedrichs1958symmetric,morawetz1954uniqueness,morawetz1958weak,morawetz1970dirichlet,morawetz2004mixed}.
For an overview of well-posedness results, see~\cite[Chapter 4, \S 17, 18]{bers2016mathematical}, \cite[Chapter X, \S 2.3]{dautray1999mathematical}.
In particular, many proofs of uniqueness (e.g.~\cite{aziz1978uniqueness,morawetz1954uniqueness,protter1953uniqueness,protter1955uniqueness}) employ the \emph{$abc$-method} of Friedrichs~\cite{friedrichs1958symmetric}, which consists in multiplying the PDE by a suitable test function of the form $a v + b v_x + c v_y$, for selected parameter functions $a,b,c$, and applying integration by parts to derive energy estimates. 
Such test functions are commonly referred to as \emph{Morawetz multipliers}.  

\paragraph{Previous numerical methods for mixed-type problems.} 
Equation~\eqref{eq:Frankl} can be rewritten equivalently as a first-order system. 
Based on Friedrichs' theory of symmetric positive first-order systems~\cite{friedrichs1958symmetric}, several numerical methods have been proposed for the first-order formulation.
These include finite difference~\cite{katsanis1969numerical}, 
least-squares~\cite{fix1978least}, and discontinuous Galerkin~\cite{huang1985discontinuous} schemes.
More generally, discontinuous Galerkin formulations for Friedrichs systems have been analyzed in~\cite{ern2006discontinuous1,ern2006discontinuous2,jensen2004discontinuous}. 
A finite element method based on the second-order formulation in the elliptic region and on first-order formulation in the hyperbolic region has also been considered in~\cite{fix1977patched}.

We focus on numerical methods applied directly to the second-order equation~\eqref{eq:Frankl}.
Among these, we recall that a finite element method for the Tricomi equation was developed in~\cite{trangenstein1977finite}, where the formulation is restricted to the elliptic subproblem. 
Subsequently, in~\cite{deacon1979finite} a finite element method was proposed for the Lavrent'ev--Bitsadze equation ($K(y)=\mathrm{sign}(y)$), combining a variational formulation in the elliptic region with a Cauchy problem in the hyperbolic region.

Another approach for treating the second-order equation consists of formulations based on the multiplier technique, motivated by the classical energy-integral method. 
In this framework, an $H^2(\Omega)$-conforming Galerkin scheme for the Tricomi problem was proposed in~\cite{aziz1980finite}. 
The authors of \cite{aziz1980finite} employed affine multipliers $b$ and $c$ and chose $a=0$ in Friedrichs’ $abc$ framework, leading to an energy inequality in $H^1(\Omega)$. 
This analysis allowed them to establish existence, uniqueness, and \emph{a priori} error estimates for the discrete solution.
No numerical results were provided.
This method was further generalized in~\cite{aziz1984finite}.
A different $H^2(\Omega)$-conforming Galerkin formulation was proposed in \cite{sermer1983galerkin}, where distinct multipliers are employed in the elliptic and hyperbolic regions.
The choice $a \neq 0$ in Friedrichs’ framework leads to coercivity in an energy norm associated with a weighted Sobolev space, larger than $H^1(\Omega)$.
Numerical experiments using bicubic splines show comparable accuracy in both the elliptic and hyperbolic regions.
In both~\cite{aziz1980finite} and~\cite{sermer1983galerkin}, continuity is established in a norm stronger than the one in which coercivity holds. This mismatch leads to convergence rates that are suboptimal with respect to the approximation properties of the finite element space; see Remark~\ref{rem:suboptimal}.

\paragraph{Features of the proposed DG method.}
In this work, we propose and analyze a discontinuous Galerkin method for the Tricomi problem~\eqref{eq:Frankl}--\eqref{eq:bndc:Dir}, considering directly the second-order equation.
Our formulation is inspired by the energy-integral method: since multipliers play a central role in the analysis of mixed-type equations at the continuous level, it seems natural to incorporate them also in the numerical scheme, as done in~\cite{aziz1980finite,sermer1983galerkin}. 
In particular, our approach is motivated by the finite element method studied in~\cite{aziz1980finite}, but it avoids the need for globally $C^1$ elements by using a discontinuous discretization. 
Following this approach, we employ suitable multipliers in the form $b v_x+c v_y$ to derive an energy estimate in a mesh-dependent norm.  

Stability of the method is ensured provided that the penalty parameter associated with the jumps of discrete functions is positive and the penalty parameters associated with the jumps of first derivatives are sufficiently large. 
The method is well-posed on general polygonal meshes and for arbitrary polynomial spaces.

Owing to the flexibility of DG methods in the choice of discrete spaces, we allow the use of spaces with a reduced number of degrees of freedom while preserving good approximation properties. 
In particular, classical Trefftz spaces are spanned by exact solutions of the homogeneous PDE, see e.g.~\cite{MoPe18}.
Instead of constructing a basis of exact solutions for the operator~\eqref{eq:Frankl}, we consider approximate solutions of the PDE, namely the quasi-Trefftz and the embedded Trefftz spaces. 
A general strategy to build quasi-Trefftz spaces for linear operators is described in~\cite{10.1093/imanum/drae094} and can be applied to~\eqref{eq:Frankl}, while the embedded Trefftz method~\cite{lehrenfeld2023embedded,lozinski19} avoids the explicit construction of Trefftz spaces. 
Quasi-Trefftz and embedded Trefftz allow comparable reduction in the discrete space dimension.

We prove \emph{a priori} error bounds in an energy norm and derive $hp$-error estimates for
standard polynomials and $h$-error estimates for quasi-Trefftz spaces. 
As is typical for mixed-type problems, the resulting convergence rates are in general suboptimal, of order $\mathcal{O}(h^{p-1})$, when polynomial degree $p$ is used.
The method is numerically robust with respect to the choice of the penalty parameters. 

\paragraph{Notation.}
Let $ D \subset \IR^2$ be an open, bounded, Lipschitz domain with boundary $\partial D$.
We denote by $L^2(D)$ the space of Lebesgue square integrable
functions on~$D$ with norm $\|\cdot\|_{L^2(D)}$ and by $H^s(D)$ the Sobolev space of order $s\in\IR$  with norm $\|\cdot\|_{H^s(D)}$.
For a Lipschitz curve $S$, we write $L^2(S)$ and $H^s(S)$ for the analogous spaces.
For $p \in \IN$, the space of polynomials of total degree at most~$p$ defined on~$D$ is denoted by $\mathbb{P}^p(D)$.
Given a sufficiently regular function $u : D \to\IR$, we use the notation $u_x$, $u_y$ for the first-order partial derivatives with respect to the Cartesian variables $x$ and $y$, respectively. Similarly, we use $u_{xx}$, $u_{yy}$, $u_{xy}$ for the second-order partial derivatives. The gradient is written as $\nabla u := (u_x, u_y)^\top$, where $(\cdot,\cdot)^\top$ indicates the transpose operator.
The outward unit normal vector on the boundary $\partial D$ is $\bn = (n_x, n_y)^\top$ and $\bt:=\bn^\perp=(-n_y,n_x)^\top$ is the unit tangential vector.
The normal and tangential components of the gradient $\nabla u$ are denoted by  $u_n:=\nabla u \cdot \bn$ and $u_t :=\nabla u \cdot \bt$, respectively.
Table~\ref{tab:notation} summarizes the main symbols used throughout the article.

\paragraph{Structure of the paper.}
The paper is organized as follows.
We introduce the discontinuous Galerkin formulation in Section~\ref{s:DG} and analyze its well-posedness in Section~\ref{s:well-posedness}.
Section~\ref{s:erroranalysis} is devoted to the derivation of \emph{a priori} error estimates in the energy norm and provides $hp$-error bounds for the standard polynomial space and $h$-error bounds for the quasi-Trefftz polynomial space.
In Section~\ref{s:numericalexperiments}, numerical experiments are presented to validate the theoretical results and to illustrate additional features of the proposed method.
Finally, we draw some conclusions in Section~\ref{s:conclusions}.

\section{Discontinuous Galerkin discretization}\label{s:DG}
In this section, we present the proposed discontinuous Galerkin method for the discretization of the model problem~\eqref{eq:Frankl}–\eqref{eq:bndc:Dir}.

\subsection{Mesh assumptions and notation}\label{s:mesh}
Let $\calT_h$ be a partition of the domain $\Omega$ into disjoint open elements $T$ such that $\overline{\Omega}=\cup_{T\in\calT_h}\overline{T}$.
We assume that interior elements, meaning that their closure intersects $\deO$ at most in a point, are polygons, while the remaining elements may have curved facets that lie exactly on $\partial \Omega$.

Each element $T\in\calT_h$ has diameter $h_T := \sup_{\bx,\by\in T} \abs{\bx-\by}$ and the global mesh size is defined as $h:= \sup_{T\in\calT_h}h_T$.
We consider a sequence of meshes $\calT_{\mathcal{H}}:=\{\calT_h\}_{h\in\mathcal{H}}$, where $\mathcal{H}\subset(0,+\infty)$ is countable with $0$ as its only accumulation point.
For each $T\in\calT_h$, denote by 
$\partial T$ its boundary and $\bn_T$ the unit outward normal vector on $\partial T$.

The intersection $F=\partial T\cap\partial T'$, for two distinct elements $T,T'\in \calT_h$, is either empty, or a point, or a straight segment, and in this case we call it an interior facet.
A boundary facet is $F=\partial T\cap \Gamma_j$ for $j\in\{0,1,2\}$ with positive 1-dimensional measure and can be curvilinear.
Distinct facets of $T$ may be co-planar; in particular, hanging nodes are allowed. 
The set of all mesh facets is written as $\calF_h$.
We denote by $\calF_T:=\{F\in\calF_h\mid F\subset\partial T\}$ the set of all facets of $T\in\calT_h$, by $\calF_h^{\mathrm I}$ the set of all interior facets of the mesh, by $\calF_h^j$, $j\in\{0,1,2\}$, the set of facets contained in $\Gamma_j$, and by $\calF_h^{\mathrm D}:=\calF_h^0\cup\calF_h^1$ the set of the boundary facets where Dirichlet conditions are assigned.
Thus $\calF_h=\calF_h^{\mathrm I}\cup\calF_h^{\mathrm D}\cup\calF_h^{\mathrm 2}=\calF_h^{\mathrm I}\cup\calF_h^{\mathrm 0}\cup\calF_h^{\mathrm 1}\cup\calF_h^{\mathrm 2}$
is the set of all facets, and all unions are disjoint.
For a facet $F\in\calF_h$, we denote by	$h_F$ its diameter.

We define the broken (elementwise) Sobolev and polynomial spaces on the mesh $\calT_h$:
\begin{align*}
	H^{m}(\calT_h):=&\{v\in L^{2}(\Omega) \mid v_{|_T} \in H^{m}(T) \quad \forall T \in \calT_h \},
	\qquad m\in\IN,\\
	\IP^p(\calT_h):=&\{v\in L^2(\Omega) \mid  v_{|_T} \in \IP^p(T)  \quad \forall T\in \calT_h \},\qquad
	p\in\IN.
\end{align*}

We adopt the standard DG notation~\cite[p.~19]{CangianiDGH2017} for averages $\mvl{\cdot}$ and jumps $\jmp{\cdot}$ of any scalar function $\varphi\in H^1(\calT_h)$ and any vector-valued function  $\bw\in [H^1(\calT_h)]^2$ across the mesh facets: 
\begin{align*}
	&\begin{cases}
		\begin{aligned}
			\mvl{\varphi}:=&\ \frac{\varphi_{|_{T_+}}+\varphi_{|_{T_-}}}2, &\qquad	
			\mvl{\bw}:=&\ \frac{\bw_{|_{T_+}}+\bw_{|_{T_-}}}2,\\
			\jmp{\varphi}:=&\ \varphi_{|_{T_+}}\bn_{T_+}+\varphi_{|_{T_-}} \bn_{T_-},&\qquad
			\jmp{\bw}:=&\ \bw_{|_{T_+}}\cdot\bn_{T_+}+\bw_{|_{T_-}} \cdot\bn_{T_-},
		\end{aligned}  \quad  \text{on }  F=\partial T_+ \cap \partial T_-,
	\end{cases}
	\\&\begin{cases}
		\begin{aligned}
			\mvl{\varphi} :=&\ \varphi_{|_T}, &\qquad	\mvl{\bw} :=&\ \bw_{|_T},\\
			\jmp{\varphi}:=&\ \varphi_{|_T}\bn_T,& \qquad \jmp{\bw}:=&\ \bw_{|_T}\cdot \bn_T,
		\end{aligned} \qquad  \text{on }  F = \partial T \cap \partial \Omega.
	\end{cases}
\end{align*}
We will use the ``DG magic formula''~\cite[Prop.~2.2.5]{perinati2023quasitrefftz}:
for all $\varphi\in H^1(\calT_h)$ and for all $\bw\in[H^1(\calT_h)]^2$,
\begin{equation}\label{eq:DG_magic}
	\sum_{T\in \calT_h} \int_{\partial  T} \bw \cdot \bn_{T} \varphi = 	\sum_{F\in \FhI}
	\int_{F} \big(\mvl{\bw} \cdot \jmp{\varphi} +  \jmp{\bw}  \mvl{\varphi}\big)
	+\int_{\partial\Omega}  \bw \cdot \bn \varphi.  
\end{equation}

We make the following assumptions on mesh sequences:
\begin{enumerate}[label=(\roman*)]
	\item\label{it:starshaped} \textit{Star-shaped property}:
	there exists $0< r_\star\leq\frac12$ such that, for all $h\in\mathcal{H}$, each $ T\in\calT_h$ is star-shaped with respect to a ball centered at some $\bx\in T$ and with radius $r_\star h_T$.
	\item\label{it:graded}
	\textit{Graded mesh} (\cite[p.~744]{arnold1982interior}):
	there exists $C_g >0$ such that,  for all $h\in\mathcal{H}$, for all $ T\in\calT_h $ and for all $F\in\calF_T$,
	\begin{equation}\label{eq:gradedmesh}
		h_T\le C_g h_F.
	\end{equation}
\end{enumerate}
The graded-mesh property~\ref{it:graded} implies local quasi-uniformity: if $T_1,T_2$ are adjacent mesh elements, i.e.\ $T_1\cap T_2\in\FhI$, then $h_{T_1}\le C_g h_{T_2}$.

The star-shaped property~\ref{it:starshaped} implies the classical shape-regularity property (e.g.\ \cite[Def.~1.38(i)]{di2011mathematical}):
\begin{equation*}
	h_T\leq C_{sr}\rho_T, \qquad \text{with } C_{sr} =r_\star^{-1},  \qquad \forall h\in\mathcal{H},\quad \forall T\in\calT_h,
\end{equation*}
where $\rho_T$ is the radius of the largest ball contained in $T$.
Moreover, the star-shaped property~\ref{it:starshaped} ensures that~\cite[Ass.~4.1]{cangiani2022hp} is satisfied with, for each $K\in\calT_h$, $F_i$ the facets in $\calF_T$, $\bx_i^0$ equal to the center of the ball mentioned in~\ref{it:starshaped}, and the parameter $c_{sh}$ in~\cite[eq.~(4.1)]{cangiani2022hp} equal to~$r_\star$.
In particular, the sub-elements $K_{F_i}$ in~\cite[Ass.~4.1]{cangiani2022hp} are disjoint (possibly) curvilinear triangles. Thus, by Lemma 4.4 of~\cite{cangiani2022hp} (see also~\cite[eq.~(28)]{10.1093/imanum/drae094} for the polygonal case), the following discrete trace inequality holds for $p\ge 1$:
$$
\N{v}_{L^2(\partial T)}^2
=\sum_{F\in\calF_T}\N{v}_{L^2(F)}^2
\le \sum_{F\in\calF_T}\frac{(p+1)(p+2)}{r_\star h_T}\N{v}_{L^2(K_F)}^2
\le\frac{6\,p^2}{r_\star h_T}\N{v}_{L^2(T)}^2
\qquad \forall v\in \IP^p(T).
$$

\begin{lemma}[Discrete trace inequality]\label{lem:traceineq}
	Let $\calT_{\mathcal{H}}$ be a mesh sequence with the star-shaped property~\ref{it:starshaped}. 
	Then,
	\begin{equation} \label{eq:discretetraceinequality}
		\N{v}_{L^2(\partial T)}\leq  \Ctr\, p\,  h_T^{-\frac12} \N{v}_{L^2(T)},
	\end{equation}
	for all $h\in\mathcal{H}$, $T\in\calT_h$, $p\in\IN$, $v\in \mathbb{P}^p(T)$.
	The bounding constant is controlled by 
	$\Ctr\le\sqrt{\frac6{r_\star}}$.
\end{lemma}

\subsection{The DG variational formulation}  \label{s:DGformulation}
The formulation is derived using the $abc$ method of Friedrichs, multiplying the PDE~\eqref{eq:Frankl} by a Morawetz multiplier.
To accommodate a quasi-Trefftz discretization, the DG scheme and its abstract error analysis are developed for a general subspace $V_h$ of the broken polynomial space $\IP^p(\calT_h)$. 
We introduce the following function spaces:
$$
V_*:= H^1(\Omega) \cap H^2(\calT_h), \qquad V_{*h} := V_* + V_h.
$$
For any $v\in V_{*h}$, the 
Morawetz multiplier of $v$ is defined as
\begin{equation}\label{eq:Multiplier}
	\calM v:= \bm\cdot\nabla v=b v_x+c v_y, \qquad \bm:=(b,c)^\top,
\end{equation}
where
$b$ and $c$ are scalar function that satisfy the following conditions:
\begin{enumerate}[label=\textbf{A\arabic*}, ref=\textbf{A\arabic*}]
	\item \label{ass:A1}	Regularity:
    \begin{equation}\label{eq:bcreg}
        b=b(x), \quad  c=c(y), \qquad b, c\in C^0(\overline{\Omega})\cap C^1(\Omega);
    \end{equation}
	\item \label{ass:A2} Positivity condition (\cite[Lemma 2.1 (iii)]{aziz1980finite}):
	\begin{equation}\label{eq:delta}
		\text{Exists } \delta>0 \text{ s.t.} \quad - K b_x  + (K c)_y\ge \delta \quad  \text{ and } \quad b_x  -c_y\ge \delta \quad \iin \Omega;
	\end{equation}
	\item \label{ass:A3} Boundary inequality on $\Gamma_2$ (\cite[Lemma 2.1 (i)]{aziz1980finite}):
	\begin{equation}\label{eq:condGamma2}
		b+c\sqrt{-K}\leq 0 \quad  \oon \Gamma_2;
	\end{equation}
	\item \label{ass:A4} Boundary inequality on $\Gamma_0$ (\cite[(2.6)]{aziz1980finite}):
	\begin{equation}\label{eq:condGamma0}
		\bm\cdot \bn\ge 0 \quad  \oon \Gamma_0.
	\end{equation}
\end{enumerate}
For $K$ as in~\eqref{eq:kappa}, Lemma 2.1 in~\cite{aziz1980finite} shows that Assumptions~\ref{ass:A1}--\ref{ass:A3} are verified if $b$ and $c$ are taken as affine functions of the form $b(x)=b_0+b_1 x$ and $c(y)=c_0+c_1 y$ with suitable coefficients $b_0,b_1,c_0,c_1\in \IR$.
Sufficient conditions on these coefficients are given in~\cite[eq.~(2.5)]{aziz1980finite}.
In Section~\ref{s:choiceMorawetz}, 
we also take $b$ and $c$ as affine functions, and derive
sufficient conditions on $b_0,b_1,c_0,c_1$, given in~\eqref{eq:suffconditions} below,
for the Tricomi
problem with $K(y)=y$ on the domain $\Omega$ of
Figure~\ref{fig:domainexperiments} used in the numerical experiments. In
particular, 
we show that, for example, $b(x)=-2+\tfrac12 x$ and $c(y)=1+\tfrac14 y$
satisfy Assumptions~\ref{ass:A1}--\ref{ass:A4} with $\delta=\tfrac14$.

For our purposes, the zero-order term of the general multiplier ($av+bv_x+c v_y$) in Friedrichs’ $abc$ framework is not needed so we take $a=0$.

Since $K=K(y)$ depends only on $y$~\eqref{eq:kappa}, the Frankl equation $\calL u =f$ can be rewritten in divergence form as 
\begin{equation*}
	\mathrm{div}(\W \nabla u)=f \quad \iin \Omega\subset\IR^2, \quad\text{with} \quad \W:= \begin{bmatrix}
		K&0\\
		0 & 1
	\end{bmatrix}.
\end{equation*}
Let $u$ be the exact solution of problem~\eqref{eq:Frankl}--\eqref{eq:bndc:Dir} and assume $u\in V_*$.
We multiply~\eqref{eq:Frankl} by the Morawetz multiplier $\calM v$, for a test function $v\in H^2\Th$,
and integrate on an element $T\in \calT_h$:
\begin{equation*}
	\int_{T} \text{div}(\W \nabla u)   \calM v = \int_{T} f   \calM v.
\end{equation*}
Applying integration by parts and summing over all elements yields
\begin{equation}\label{eq:intbypart}
	-\sum_{T\in\calT_h}	\int_T  \W \nabla u \cdot \nabla  (\calM v)+
	\sum_{T\in\calT_h}	\int_{\partial T}   \W \nabla u \cdot \bn_T \calM v
	=\sum_{T\in\calT_h}	\int_T f \calM v.
\end{equation}
Using the ``DG magic formula''~\eqref{eq:DG_magic}, the second term of~\eqref{eq:intbypart} can be expressed as a sum over mesh facets:
\begin{equation*}
	\sum_{T\in \calT_{h}} \int_{\partial T}  \W \nabla u\cdot \bn_{T} \calM v = 	\sum_{F\in \FhI}
	\int_{F} \left(\mvl{\W \nabla u} \cdot \jmp{\calM v} +  \jmp{\W \nabla u}  \mvl{\calM v}\right)
	+	\sum_{F\in \FhD\cup\Fhtwo}	\int_{F}  \W \nabla u\cdot \bn \calM v.
\end{equation*}
Since $u$ satisfies~\eqref{eq:Frankl} and $f\in L^2(\Omega)$, then $\W \nabla u\in H(\mathrm{div};\Omega):=\{\bv\in [L^2(\Omega)]^2\mid \mathrm{div}(\bv)\in L^2(\Omega)\}$, implying that the jump $ \jmp{\W \nabla u}$ vanishes on interior facets.

To achieve the discrete coercivity, we add some stabilization terms.
Let $\gamma_i > 0$, for $i = 1,2,3$, be three dimensionless penalty parameters, we add on the left-hand side of~\eqref{eq:intbypart} the following terms:
\begin{align}\label{eq:penalty}
\sum_{F\in \FhI\cup \FhD} \frac{\gamma_1}{h_F^3}	\int_{F}  \jmp{u}\cdot \jmp{v}
+\sum_{F\in \FhI}\frac{  \gamma_2 p^2}{h_F}	\int_{F} ( \jmp{u_x}\cdot \jmp{v_x}+\jmp{u_y}\cdot \jmp{v_y})
+\sum_{F\in\FhD} \frac{  \gamma_3 p^2}{h_F}\int_F u_tv_t.
\end{align}
Recall that $u_t$ and $v_t$ are the tangential derivatives of $u$ and $v$ on each facet.
The terms of~\eqref{eq:penalty} on interior facets $F\in\FhI$ are consistent if $u$ is sufficiently regular, while the terms on the boundary facets are not. Hence, we also add the same terms on the boundary to the right-hand side of~\eqref{eq:intbypart} in order to maintain the consistency and then use the Dirichlet condition~\eqref{eq:bndc:Dir}.

The discontinuous Galerkin variational formulation of the problem~\eqref{eq:Frankl}--\eqref{eq:bndc:Dir} reads:
\begin{equation}\label{eq:variational}
	\text{Find } u_h \in V_h \text{ such that }
	\calA_h(u_h,v_h)+\calA_J(u_h,v_h)=L_h(v_h) \quad \forall v_h \in V_h,
\end{equation}
where the DG bilinear forms
$\calA_h: V_{*h} \times V_h \to \mathbb{R} $ and $\calA_J: V_{*h} \times V_h \to \mathbb{R} $ are defined by
\begin{align}
	\label{eq:Ah}
	\calA_h(u,v):=&
	-\sum_{T\in\mathcal{T}_{h}}	\int_{T} \W\nabla u \cdot \nabla  (\calM v)+\sum_{F\in\FhI}
	\int_{F}  \mvl{\W \nabla u } \cdot \jmp{\calM v} 
	+\sum_{F\in\FhD\cup\Fhtwo}\int_{F}  \W \nabla u\cdot \bn \calM v,  \\ 
	\calA_J(u,v):=& 	\label{eq:AJ}
	\sum_{F\in\FhI\cup\FhD}
\frac{ \gamma_1}{h_F^3}		\int_{F}  \jmp{u}\cdot \jmp{v}+\sum_{F\in \FhI} \frac{ \gamma_2 p^2}{h_F}	\int_{F} ( \jmp{u_x}\cdot \jmp{v_x}+  \jmp{u_y}\cdot \jmp{v_y})+\sum_{F\in\FhD}  \frac{ \gamma_3 p^2}{h_F}\int_F u_tv_t,
\end{align}
and the linear form $L_h:V_h\to\mathbb{R}$ is
\begin{equation*}
	L_h(v)=	\sum_{T\in\mathcal{T}_{h}}	\int_{T}  f  \calM v +	\sum_{F\in\FhD}\frac{\gamma_1}{h_F^3}	 \int_{F}g v+\sum_{F\in\FhD} \frac{ \gamma_3 p^2}{h_F}\int_F g_tv_t.
\end{equation*}

\section{Well-posedness of the DG method }\label{s:well-posedness}
In this section, we establish the well-posedness of the discrete discontinuous Galerkin problem~\eqref{eq:variational}. We adopt the framework of nonconforming methods following~\cite[Thm.~1.35]{di2011mathematical}.

If the solution $u\in H^2(\Omega)$, by construction, the variational problem~\eqref{eq:variational} is consistent, i.e.\ the solution $u\in H^2(\Omega)$ of the boundary value 
problem~\eqref{eq:Frankl}--\eqref{eq:bndc:Dir} solves
$\calA_h(u,v_h)+\calA_J(u,v_h) = L_h(v_h)$ for all $v_h \in V_h$.

\subsection{Mesh-dependent norms}\label{s:Norms}	
For all $v\in V_{*h}$ we define two mesh-dependent norms:
the energy norm
\begin{align}\label{eq:energynorm}
	\vertiii{v}^2:=\sum_{T\in\calT_h}\int_T  \delta (v_x^2 + v_y^2)+\abs{v}^2_{J},
\end{align}
where $\delta>0$ is defined in~\eqref{eq:delta} and $\abs{\cdot}_J$ is the jump seminorm given by
\begin{equation}\label{eq:jumpseminorm}
	\abs{v}^2_{J}:=
	\calA_J(v,v)=\sum_{F\in \FhI \cup \FhD} \frac{\gamma_1}{h_F^3}\int_{F} |\jmp{v}|^2
	+\sum_{F\in \FhI}\frac{\gamma_2 p^2}{h_F}\int_{F} (|\jmp{v_x}|^2+|\jmp{v_y}|^2)
	+\sum_{F\in\FhD}\frac{\gamma_3 p^2}{h_F}\int_F v_t^2,
\end{equation}
and the ``residual'' norm
\begin{equation}\label{eq:residualnorm}
\vertiii{v}^2_{\calL}:=\N{\calL v}^2_{L^2 \Th}+\abs{v}_J^2=\sum_{T\in\calT_h}\int_T  (K v_{xx} + v_{yy})^2+\abs{v}_J^2.
\end{equation}
\begin{proposition}
	$\vertiii{\cdot}$ and $ \vertiii{\cdot}_\calL $ are norms on $V_{*h}$.
\end{proposition}
\begin{proof}
	Let $v \in V_{*h}$ such that $\vertiii{v}=0$. 
	By assumption~\eqref{eq:delta} with $\delta>0$ and $\vertiii{v}=0$ we have that $\nabla v = \boldsymbol{0} $ in each element $T\in\calT_h$, which implies that \( v \) is elementwise constant. 
	Moreover, since $\jmp{v} = 0$ on all interior facets $\FhI$, $v$ is constant on $\Omega $. 
	The condition $v=0$ on $\FhD$ implies that $v = 0$ on the Dirichlet boundary. Hence, $ v = 0$ on $\Omega$.
	Now consider $v \in V_{*h}$ such that $\vertiii{v}_\calL=0$, then $\calL v = 0$ in each mesh element. Since $\jmp{v} = 0$, $\jmp{v_x} = 0$ and $\jmp{v_y} = 0$ on all interior facets, then $v\in H^2(\Omega)$, implying  $\calL v=0$ on $\Omega $. Moreover, the condition $v=0$ on $\FhD$ implies that $v = 0$ on the Dirichlet boundary. Hence, using the well-posedness of the  continuous problem~\eqref{eq:Frankl}--\eqref{eq:bndc:Dir}, we have $v = 0$ on $\Omega$.  
\end{proof}

\begin{remark}
The whole analysis can also be done with the alternative energy norm
$$ 
\sum_{T\in\calT_h}\int_T \Big(  (- K b_x  + (K c)_y ) v_x^2 + (b_x  -c_y) v_y^2\Big)+\abs{v}^2_{J},$$ 
which, thanks to the assumption~\eqref{eq:delta} on $b$ and $c$, is bounded below by $\vertiii{v}^2$.
\end{remark}

To simplify the notation we write
\begin{equation*}
	\int_{\calT_h}:=\sum_{T\in \calT_h}\int_T, \qquad  	
	\int_{\partial \calT_h}:=\sum_{T\in \calT_h}\int_{\partial T}, \qquad 
	\int_{\Fh^\bullet}:=\sum_{F\in \Fh^\bullet}\int_F \qquad\bullet\in\{\mathrm{I},\mathrm{D},0,1,2\}.
\end{equation*}
Similarly, we use $\|\cdot\|_{L^2(\calT_h)}^2:=\sum_{T\in\calT_h}\|\cdot\|_{L^2(T)}^2$.

\subsection{Discrete coercivity}
\begin{proposition}[Discrete coercivity]\label{prop:coercivity}
	Let 
	\begin{align}\label{eq:BetaGammaStar}
		\beta:=\max{\{\N{Kb}_{L^\infty(\Omega)},\N{b}_{L^\infty(\Omega)},\N{Kc}_{L^\infty(\Omega)},\N{c}_{L^\infty(\Omega)} \}},
		\;\;\text{and}\quad
		\gamma_*:=288\,\beta^2\,  \Ctr^2\,\delta^{-1},
	\end{align}
	with $C_{tr}$ defined in~\eqref{eq:discretetraceinequality} and $\delta$ in~\eqref{eq:delta}.
	For all $\gamma_1>0$ and for all $\gamma_2,\gamma_3\ge\gamma_*$, 
	we have
	\begin{equation*}
		\calA_h(v,v)+\calA_J(v,v)\ge \frac14 \vertiii{v}^2  \quad \forall v\in V_h.
	\end{equation*}
\end{proposition}

To prove Proposition~\ref{prop:coercivity} we first show some preliminary results.

\begin{lemma}\label{lem:Ah+AJ}
	For all $v\in V_h$, the bilinear forms $\calA_h$~\eqref{eq:Ah} and $\calA_J$~\eqref{eq:AJ} satisfy:
	\begin{equation}\label{eq:rewritecoerc}
		\begin{split}
			\calA_h(v,v)+\calA_J(v,v)=&	\frac12 \int_{\calT_h}\Big( v_x^2(- K b_x  +(Kc)_y  )+v_y^2(b_x -  c_y)\Big)\\&+
			\int_{\FhI} \left( \mvl{ \W\nabla v } \cdot \jmp{\calM v} -\frac12	  \mvl{\bm
			} \cdot \jmp{Kv _x^2+v_y^2} \right)\\ 
			& 	+	\int_{\deO}
			\left( \W\nabla v\cdot \bn \calM v 
			-\frac12\bm\cdot \bn (Kv_x^2+v_y^2)\right)\\&
			+\int_{\FhI\cup\FhD}\frac{\gamma_1}{h_F^3}|\jmp{v}|^2
			+\int_{\FhI} \frac{\gamma_2 p^2}{h_F}(|\jmp{v_x}|^2+|\jmp{v_y}|^2)
			+\int_{\FhD} \frac{ \gamma_3 p^2}{h_F} v_t^2.
		\end{split}
	\end{equation}	
\end{lemma}
\begin{proof}
	Let $v\in V_h$.  By integrating by parts the volume term in $\calA_h$~\eqref{eq:Ah} and using the regularity~\eqref{eq:bcreg}, \eqref{eq:kappa} of $b$, $c$ and $K=K(y)$, we rewrite
	\begin{align*}
		-\int_{\calT_h}  \W \nabla v \cdot \nabla  (\calM v)=&-\int_{\calT_h} \Big(K v_x (b_x v_x + b v_{xx}+c v_{yx})+v_y (b v_{xy} + c_y v_y+c v_{yy})\Big)\\=&
		-\int_{\calT_h} \Big(K b_x v_x^2 + K b  \frac12 (v_x^2)_x+ K c \frac12(v_x^2)_y+b\frac12 (v_y^2)_x + c_y v_y^2+c  \frac12(v_y^2)_y\Big)\\=&
		\int_{\calT_h} \Big( -K b_x v_x^2 + K b_x \frac12 v_x^2+( Kc)_y \frac12 v_x^2+b_x\frac12 v_y^2 - c_y v_y^2+c_y \frac12 v_y^2\Big)\\
		&	+\int_{\partial \calT_h}  \Big(-K b \frac12 v_x^2 n_x-K  c \frac12 v_x^2 n_y-b \frac12 v_y^2 n_x-c \frac12 v_y^2n_y\Big)\\=&
		\frac12 \int_{\calT_h}  \Big(v_x^2(- K b_x  +(Kc)_y  )+v_y^2(b_x -  c_y)\Big)	-	\frac12 \int_{\partial \calT_h}  (Kv_x^2+v_y^2)( b n_x+c  n_y).
	\end{align*}
	The last term can be rewritten using the ``DG magic formula''~\eqref{eq:DG_magic} and the continuity~\eqref{eq:bcreg} of $b$ and $c$ (so that $\jmp{\bm}=0$) as
	\begin{align*}
		-\frac12  \int_{\partial \calT_h}(Kv_x^2+v_y^2)( b n_x+c  n_y)= -\frac12	\int_{\FhI}	\mvl{\bm} \cdot \jmp{Kv_x^2+v_y^2} 
		-\frac12	\int_{\FhD\cup\Fhtwo}  \bm\cdot \bn (Kv_x^2+v_y^2).
	\end{align*}
	Putting the above expressions together, and recalling the definitions~\eqref{eq:Ah}--\eqref{eq:AJ}, we get~\eqref{eq:rewritecoerc}.
\end{proof}
The next lemma expresses the term integrated over $\deO$ in~\eqref{eq:rewritecoerc} in terms of normal and tangential derivatives of $v$.
\begin{lemma}\label{lem:Qdecomposition}
	For all $v\in V_h$, the following equality holds
	\begin{equation*}
		(\W \nabla v)\cdot \bn\, \calM v
		-\frac12\bm\cdot \bn (Kv_x^2+v_y^2)
		=  
		\frac12(v_n^2 Q_n+
		v_t^2 Q_t+
		2v_n v_t Q_{nt})
		\qquad \oon \deO,
	\end{equation*}
	where 
	\begin{align}\label{eq:QnQtQnt}
		Q_n := (K n_x^2+n_y^2)( b n_x+c n_y), \qquad 
		Q_t :=\bt^\top \mathbf{M} \bt \qquad
		Q_{nt}:=(K n_x^2+n_y^2)(b t_x+c t_y),
	\end{align}
	and $\mathbf{M}\in\IR^{2\times 2}$ is the symmetric matrix (introduced in~\cite[pag. 17]{aziz1980finite})
	\begin{align*}
		\mathbf{M}=
		\begin{pmatrix}
			K (bn_x-cn_y)& b n_y+Kcn_x\\
			b n_y+Kcn_x& -(bn_x- cn_y)\\
		\end{pmatrix}.
	\end{align*} 
\end{lemma}
\begin{proof}
	Expanding the left-hand side yields
	\begin{align*}
		(\W \nabla v)\cdot \bn\, \calM v  -&\frac12 \bm \cdot \bn (Kv_x^2+v_y^2)
		=(Kv_x  n_x+v_y n_y) (bv_x+cv_y)-\frac12 (b n_x+c n_y)  (Kv_x^2 +v_y^2)\\
		&=\frac12 Kb v_x^2 n_x+\frac12 c v_y^2  n_y+b v_y v_x n_y+Kc v_x v_y n_x-\frac12 b v_y^2 n_x-\frac12 K cv_x^2 n_y\\
		&=\frac12 \Big((K v_x^2-v_y^2)( b n_x-c n_y)+2v_xv_y(bn_y+Kc n_x)\Big).
	\end{align*}
	We decompose the gradient into the normal and tangential components: $\nabla v=(v_x,v_y)^\top= v_n \bn+v_t \mathbf{t}$, where
	$\bn=(n_x,n_y)^\top$ is the outward unit normal on the boundary $\partial \Omega$ and $\bt:=\bn^{\perp}=(t_x,t_y)^\top=(-n_y,n_x)^\top$, so that 
	\begin{align*}
		v_x^2&= v_n^2 n_x^2+v_t^2 t_x^2+2v_nv_t n_x t_x,\\
		v_y^2 &= v_n^2 n_y^2 + v_t^2 t_y^2+2 v_nv_t n_y t_y,\\
		v_x v_y &= v_n^2 n_x n_y + v_t^2 t_x t_y+ v_n v_t (n_x  t_y+  t_x  n_y).
	\end{align*}
	Plugging this into the above expression and regrouping terms, we obtain
	\begin{align*}
		&\frac12\Big(K (v_n^2 n_x^2+v_t^2 t_x^2+2v_nv_t n_x t_x)-(v_n^2 n_y^2 + v_t^2 t_y^2+2 v_nv_t n_y t_y)\Big)( b n_x-c n_y)\\ 
		& +(v_n^2 n_x n_y + v_t^2 t_x t_y+ v_n n_x v_t t_y+ v_t t_x v_n n_y)(bn_y+Kc n_x)\\
		=& \frac12 v_n^2\Big((Kn_x^2-n_y^2)( b n_x- c n_y)+2n_x n_y(bn_y+Kc n_x)\Big) \\
		&+\frac12 v_t^2\Big((Kt_x^2-t_y^2)(  b n_x- c n_y)+2t_x t_y(bn_y+Kc n_x)\Big) \\
		&+\frac12v_n v_t\Big(2(K n_x t_x-n_y t_y)(  b n_x-  c n_y)+2(n_x t_y+t_x n_y)(bn_y+Kc n_x)\Big) \\
		=&\frac12v_n^2(K bn_x^2 n_x+K cn_x^2 n_y+ c  n_y^2n_y+b n_y^2n_x) 
		+\frac12v_t^2\Big((Kt_x^2-t_y^2)(  b n_x- c n_y)+2t_x t_y(bn_y+Kc n_x)\Big)\\&
		+v_n v_t(K  b n_x^2 t_x+c n_y^2 t_y+bn_y^2 t_x+Kc n_x^2 t_y) \\
		=&\frac12  v_n^2(K n_x^2+n_y^2)( b n_x+c n_y) 
		+\frac12 	v_t^2\big( \bt^\top\mathbf{M} \bt \big) 
		+v_n v_t (K n_x^2+n_y^2)(b t_x+c t_y).
	\end{align*}
	This expression can be rewritten as
	$\frac12(v_n^2 Q_n+v_t^2 Q_t+2v_n v_t Q_{nt})$
	using the definitions in~\eqref{eq:QnQtQnt}.
\end{proof}

Next, we prove that the boundary integral over the characteristic arc $\Gamma_2$ in the bilinear form~\eqref{eq:rewritecoerc} is non-negative, following the approach of~\cite[Theorem 2.1]{aziz1980finite}.
\begin{lemma}\label{lem:Gamma2pos}
	For all $v\in V_h$, the following inequality holds
	\begin{equation*}
		\int_{\Fhtwo}  \left(\W \nabla v\cdot \bn \calM v
		-\frac12\bm\cdot \bn (Kv_x^2+v_y^2)\right)\ge 0.
	\end{equation*}
\end{lemma}
\begin{proof} 
	The characteristic relation $K n_x^2+n_y^2=0$~\eqref{eq:relationchar} on $\Gamma_2$ implies $Q_n=Q_{nt}=0$. Hence, the integral reduces to
	\begin{align*}
		&	\int_{\Fhtwo}  \Big(\W\nabla v\cdot \bn \calM v
		-\frac12 \bm \cdot \bn (Kv_x^2+v_y^2)\Big)= \frac12 	\int_{\Fhtwo}  
		v_t^2 \bt^\top \mathbf{M} \bt.
	\end{align*}
	To conclude, we show that $\mathbf{M}$ is positive semidefinite. Its determinant satisfies
	\begin{align*}
		\mathrm{det}(\mathbf{M})&=-K (bn_x-cn_y)^2-( b n_y+Kcn_x)^2
		=
		(Kn_x^2+n_y^2)(-b^2-Kc^2)=0,
	\end{align*}
	using again the characteristic relation~\eqref{eq:relationchar}.
	Its trace is
	\begin{align*}
		\mathrm{trace}(\mathbf{M})&
		=(K-1) (bn_x-cn_y)
		=-\sqrt{1-K}(b+c\sqrt{-K}),
	\end{align*}
	where we used the explicit expression  of the outer normal on $\Gamma_2$:
	\begin{equation*}
		n_y = -\frac{\sqrt{-K}}{\sqrt{1-K}},\qquad	n_x= \frac{1}{\sqrt{1-K}},
	\end{equation*}
	obtained combining the normalization $n_x^2+n_y^2=1$ with the characteristic relation \eqref{eq:relationchar}.
	Under the assumption~\eqref{eq:condGamma2} that $b+c\sqrt{-K}\leq 0$ on $\Gamma_2$, it follows that  $\mathrm{trace}(\mathbf{M})\ge 0$.
	Since $\mathbf{M}$ is a $2\times 2$ matrix with zero determinant and non-negative trace, it is positive semidefinite. 
	In particular, $\bt^\top  \mathbf{M} \bt\ge 0$, implying the assertion.
\end{proof}

Thanks to Assumption~\eqref{eq:delta},
the expression~\eqref{eq:rewritecoerc}, 
Lemma~\ref{lem:Gamma2pos}, 
and the norm definitions~\eqref{eq:energynorm}--\eqref{eq:jumpseminorm},
for all $v\in V_h$ the bilinear forms~\eqref{eq:Ah}--\eqref{eq:AJ} satisfy
\begin{align}\label{eq:first}
	\calA_h(v,v)+\calA_J(v,v)\ge
	&\ \frac12 \int_{\calT_h}\delta ( v_x^2+v_y^2)+\abs{v}_J^2 
	+\int_{\FhI} \left( \mvl{ \W\nabla v } \cdot \jmp{\calM v} -\frac12\mvl{\bm} \cdot \jmp{Kv _x^2+v_y^2} \right)\\ 
	&+\int_{\FhD} \left( \W\nabla v\cdot \bn \calM v 
	-\frac12\bm \cdot \bn (Kv_x^2+v_y^2)\right).\nonumber
\end{align}
In the next two lemmas, we treat internal and boundary facets separately, deriving bounds for each.

\begin{lemma}[Internal facets term]\label{lem:Internalfacets}
	For all $v\in V_h$, the following inequality holds
	\begin{equation*}
		\abs{	\int_{\FhI}  \mvl{ \W \nabla v} \cdot \jmp{\calM v} -\frac12	\int_{\FhI}  \mvl{\bm} \cdot \jmp{Kv_x^2+v_y^2} } \leq	\sqrt2\beta \Ctr \gamma_2^{-\frac12}  \delta^{-\frac12}
		\abs{v}_J \vertiii{v}.
	\end{equation*}
\end{lemma}
\begin{proof}
	Let $F\in\FhI$ be an internal facet and let $T^+, T^-\in \calT_h$ be two distinct elements such that $F=\partial T^+\cap \partial T^-$. Denote by $v^+$ and $v^-$ the traces of $v$ from the elements $T^+$ and $T^-$, respectively.
	We introduce the notation 
	\begin{align}\label{eq:jumpxy}
		\jmp{v}_x:=v^+n_x^++v^-n_x^-,\qquad \jmp{v}_y:=v^+n_y^++v^-n_y^- \qquad \oon F=\partial T^+\cap\partial T^-.
	\end{align}
	Note that $\jmp v$ is a vector normal to $F$, while $\jmp v_x$ and $\jmp v_y$ are scalar quantities.
	We rewrite the integral over a single facet as:
	\begin{align*}
		&\int_{F}  \mvl{ \W \nabla v} \cdot \jmp{\calM v} -\frac12	\int_{F}  \mvl{\bm} \cdot \jmp{Kv_x^2+v_y^2} \\
        &\quad =\int_{F} \frac12  \Big((K (v_x^++v_x^-))n_x^++ (v_y^++v_y^-)n_y^+\Big) \Big(b (v_x^+-v_x^-)+c (v_y^+- v_y^-)\Big)\\&
		\qquad -\frac12\int_{F}  \Big(b n_x^++c n_y^+\Big) \Big(K(v_x^+)^2+(v_y^+)^2-K(v_x^-)^2-(v_y^-)^2\Big)
		\\ &\quad =
		\int_{F} \frac12  \Big((K (v_x^++v_x^-)n_x^+c (v_y^+- v_y^-)+ (v_y^++v_y^-)n_y^+b (v_x^+-v_x^-)\Big)\\&\qquad -\frac12	\int_{F}  \Big(c n_y^+K(v_x^+-v_x^-)(v_x^++v_x^-)+b n_x^+(v_y^+-v_y^-)(v_y^++v_y^-)\Big)
        \\ &\quad =
		\frac12\int_{F}  \Big((v_y^+- v_y^-) n_x^+-(v_x^+-v_x^-)n_y^+\Big) \Big(Kc (v_x^++v_x^-)-b (v_y^++v_y^-)\Big)
        \\ &\quad=
		\int_{F} \Big(\jmp{v_y}_x-\jmp{v_x}_y\Big)\mvl{Kcv_x-b v_y}.
	\end{align*}
		Using the Cauchy--Schwarz inequality, the inequalities $(a_1+a_2)^2\leq 2(a_1^2+a_2^2)$ and $\jmp{v_y}_x^2\leq |\jmp{v_y}|^2$ ($n_x^2\leq 1$), we obtain
		\begin{align*}
			&\abs{\int_{F} \Big(\jmp{v_y}_x-\jmp{v_x}_y\Big)\mvl{Kcv_x-b v_y} } \\&\quad\leq
			\max{\{\N{Kc}_{L^{\infty}(\Omega)},\N{b}_{L^{\infty}(\Omega)}\}} 
			\left(	\int_{F}  \Big(\jmp{v_y}_x -\jmp{v_x}_y\Big)^2\right)^\frac12   	\left(	\int_{F}  \Big(\mvl{v_x} -\mvl{v_y}\Big)^2\right)^\frac12\\&\quad\leq	\beta
			\left(	\int_{F}  2\Big(\jmp{v_y}_x^2 +\jmp{v_x}_y^2\Big)\right)^\frac12   	\left(	\int_{F} 2 \Big(\mvl{v_x}^2 +\mvl{v_y}^2\Big)\right)^\frac12\\&\quad
			\leq\sqrt{2}\beta
			\left(\int_F\frac{\gamma_2 p^2}{h_F}\Big(|\jmp{v_y}|^2 +|\jmp{v_x}|^2\Big)\right)^\frac12  
			\left(\int_F\frac{h_F}{\gamma_2 p^2}\Big((v_x^+)^2+ (v_x^-)^2 +(v_y^+)^2+(v_y^-)^2\Big)\right)^\frac12.
		\end{align*}
		Summing over all internal facets, applying the Cauchy--Schwarz inequality again, recalling the jump seminorm~\eqref{eq:jumpseminorm}, using $h_F\leq h_T$ for all $F\in \calF_T$, and the discrete trace inequality~\eqref{eq:discretetraceinequality}, we obtain
		\begin{align*}
			\bigg|\int_{\FhI} &  \mvl{ \W \nabla v} \cdot \jmp{\calM v} - 
			\frac12\int_{\FhI}  \mvl{\bm} \cdot \jmp{Kv_x^2+v_y^2} \bigg|\\
            &\leq\sqrt{2}\beta\abs{v}_J
			\left(\int_{\FhI} \frac{h_F}{\gamma_2 p^2}  \Big((v_x^+)^2+ (v_x^-)^2 +(v_y^+)^2+(v_y^-)^2\Big)\right)^\frac12
			\\
			&\leq\sqrt{2}\beta\abs{v}_J
			\bigg(\sum_{T\in\calT_h} \sum_{F\in\calF_T} \frac{h_F}{\gamma_2p^2} \int_{F} (v_x^2+v_y^2)\bigg)^\frac12
			\\
			&\leq \sqrt2\beta \abs{v}_J
			\Bigg(\sum_{T\in\calT_h}\frac{h_T}{\gamma_2 p^2} \int_{\partial T} (v_x^2+v_y^2)\bigg)^\frac12\\
            &\leq \sqrt2\beta \abs{v}_J
			\Bigg(\sum_{T\in\calT_h}    \frac{\Ctr^2}{\gamma_2}\int_{T}  (v_x^2+v_y^2)\bigg)^\frac12.
		\end{align*}
		The assertion follows recalling the definition~\eqref{eq:energynorm} of the norm $\vertiii{\cdot}$.
	\end{proof}

	\begin{lemma}[Dirichlet facets term]\label{lem:Gamma01}
		For all $v\in V_h$, the following inequality holds
		\begin{equation*}
		\abs{\int_{\FhD}  \left(\W \nabla v\cdot \bn \calM v-\frac12\bm\cdot \bn (Kv_x^2+v_y^2)\right)}
		\leq2\sqrt2\,\beta\, \Ctr\, \delta^{-\frac12}\, \gamma_{3}^{-\frac12}\, \abs{v}_J\,\vertiii{v}.
		\end{equation*}
	\end{lemma}
	\begin{proof}
		Lemma~\ref{lem:Qdecomposition} allows to write the integral in the assertion as 
		$\int_{\Gamma_0\cup\Gamma_1}\frac12(v_n^2Q_n+v_t^2Q_t+2v_nv_tQ_{nt})$, with the $Q_\bullet$ terms defined in~\eqref{eq:QnQtQnt}.
		The characteristic relation $K n_x^2+n_y^2=0$ in~\eqref{eq:relationchar} implies $Q_n=Q_{nt}=0$ on $\Gamma_1$.
		The assumption $\bm\cdot \bn\ge0$ on $\Gamma_0$ made in~\eqref{eq:condGamma0} implies $v_n^2Q_n\ge 0$ on $\Gamma_0$.
		Thus
		\begin{equation}\label{eq:IntFD}
			\int_{\FhD}  \left(\W \nabla v\cdot \bn \calM v-\frac12\bm\cdot \bn (Kv_x^2+v_y^2)\right)
			\ge \int_{\FhD} \frac12v_t^2 Q_t
			+\int_{\Fhzero}v_n v_t Q_{nt}.
		\end{equation}
		Let $F\in\Fhzero$ be a Dirichlet boundary facet on the elliptic boundary, and let $T\in \calT_h$ be the element such that $F=\partial T\cap \Gamma_0$.
		Using the definitions~\eqref{eq:QnQtQnt} and~\eqref{eq:BetaGammaStar} of $Q_{nt}$ and of $\beta$,
		the Cauchy--Schwarz inequality, 
		$n_x^2+n_y^2=1$, $|t_x|+|t_y|\le\sqrt2$, $|v_n|\le|\nabla v|$,
		we obtain
		\begin{equation}	\label{eq:vnvtQnt}
			\abs{\int_{F}  v_nv_tQ_{nt}}
			= \abs{\int_{F} v_n v_t(K n_x^2+n_y^2)(b t_x+c t_y)}
			\le \beta\int_F |v_nv_t| (|t_x|+|t_y|)
			\le \sqrt2\,\beta \|v_t\|_{L^2(F)} \|\nabla v\|_{L^2(F)}.
		\end{equation}
		Summing over the facets on $\Gamma_0$ and applying the Cauchy–Schwarz inequality again, together with $h_F\le h_T$ for all $F\in \calF_T$ and the discrete trace inequality~\eqref{eq:discretetraceinequality},
		\begin{align*}
		    \abs{\int_{\Fhzero} v_n v_t Q_{nt}}
		    &\le \sqrt2\, \beta\, (\delta\gamma_3)^{-\half}
		    \sum_{F\in\Fhzero} \frac{\gamma_3^\half p}{h_F^\half}  \|v_t\|_{L^2(F)}
		    \frac{(\delta h_F)^\half}{p}\|\nabla v\|_{L^2(F)}
		    \\
		    &\le\sqrt2\, \beta\, (\delta\gamma_3)^{-\half}
		    \Bigg(\sum_{F\in\Fhzero} \frac{\gamma_3p^2}{h_F}  \|v_t\|_{L^2(F)}^2\Bigg)^\half
		    \Bigg(\sum_{F\in\Fhzero}\frac{\delta h_F}{p^2}\|\nabla v\|_{L^2(F)}^2\Bigg)^\half
		    \\
		    &\le\sqrt2\, \beta\, (\delta\gamma_3)^{-\half} \abs{v}_J
		    \Bigg(\sum_{T\in\calT_h}\Ctr^2\delta \|\nabla v\|_{L^2(T)}^2\Bigg)^\half
		    \\
		    &\le\sqrt2\, \beta\, \Ctr\, (\delta\gamma_3)^{-\half} \abs{v}_J\vertiii{v}.
		\end{align*}
		Let $F\in\FhD$ be a Dirichlet boundary facet and $T\in \calT_h$ the element such that $F=\partial T\cap \Gamma_0$ or  $F=\partial T\cap \Gamma_1$. 
		Using $|v_x t_x+v_yt_y|\le|\bt||\nabla v|$, $2t_xt_y\le|\bt|^2=|\bn|^2=1$ and again the Cauchy--Schwarz inequality, we have
		\begin{align*}
			\abs{\int_F\frac12 v_t^2 Q_t}&=\abs{\int_F\frac12 v_t (v_x t_x+v_yt_y)\Big((Kt_x^2-t_y^2)(  b n_x- c n_y)+2t_x t_y(bn_y+Kc n_x)\Big)}
			\\&
			\leq\beta \int_F\frac12| v_t| |\bt||\nabla v| \Big( (|t_x^2|+|t_y^2|)(|n_x|+|n_y|)+|n_y|+ |n_x|\Big)
			\\
			&\leq\beta \int_F | v_t| |\bt||\nabla v| \Big( |n_x|+|n_y|\Big)
			\\
			&\le\sqrt2\, \beta \|v_t\|_{L^2(F)} \|\nabla v\|_{L^2(F)}.
		\end{align*}
		Comparing with~\eqref{eq:vnvtQnt}, we see that this term can be bounded as the other one in~\eqref{eq:IntFD} and the assertion follows.
	\end{proof}
	\begin{proof}[Proof of Proposition~\ref{prop:coercivity}]
		Putting together the bound~\eqref{eq:first} and Lemmas~\ref{lem:Internalfacets} and~\ref{lem:Gamma01}, we get
	\begin{align*}
			\calA_h(v,v)+\calA_J(v,v)
		& \ge\frac12 \sum_{T\in\mathcal{T}_{h}}	\int_{T} \delta (v_x^2+v_y^2)+\abs{v}_J^2
				-\sqrt{2}\beta \Ctr \gamma_2^{-\frac12}  \delta^{-\frac12}\vertiii{v}^2
			-
			2\sqrt2\beta\Ctr \gamma_3^{-\frac12}\delta^{-\frac12}\vertiii{v}^2
			\\&\ge 
			\left(\frac12 -\sqrt2(\gamma_2^{-\frac12}+2\gamma_3^{-\frac12})\beta\Ctr  \delta^{-\frac12} \right) \vertiii{v}^2. 
		\end{align*}
		We have coercivity if the bracket is positive, which is ensured by taking sufficiently large penalty parameters.
		For example, if $\min\{\gamma_2,\gamma_3\}\ge\gamma_*$ with $\gamma_*$ as in~\eqref{eq:BetaGammaStar}, the thesis follows.
	\end{proof}
	
	\begin{corollary}[Existence and uniqueness of a discrete solution]\label{cor:WELL-POSEDNESS}
		Under the assumptions on the boundary value problem and the mesh made in Sections~\ref{s:introduction} and~\ref{s:DG}, let $\gamma_*$ be as in~\eqref{eq:BetaGammaStar}
		with $\Ctr$ defined in~\eqref{eq:discretetraceinequality} and $\delta$ in~\eqref{eq:delta}.
		For all $\gamma_1>0$, for all $\gamma_2,\gamma_3\ge\gamma_*$ and any polynomial discrete space $V_h\subset \IP^p(\calT_h)$, there exists a unique solution $u_h\in V_h$ to the DG variational formulation~\eqref{eq:variational}.
	\end{corollary}
	
	\subsection{Boundedness}
	\begin{proposition}[Boundedness]\label{prop:continuity}
		There exists $M>0$, independent of $h$ and $p$, such that 
		\begin{equation*}
			\calA_h(v,w)+\calA_J(v,w)\leq M\vertiii{v}_{\calL} \vertiii{w} \quad \forall (v,w)\in V_{*h}\times V_h,
		\end{equation*}
	with the norms defined in \eqref{eq:energynorm} and \eqref{eq:residualnorm}.
	In particular, we can take
		$M:=\sqrt{2}\,\beta\,\delta^{-\frac12}\,(1+\Ctr\gamma_2^{-\frac12})+1$.
	\end{proposition}
	\begin{proof}
		Let $(v,w)\in V_{*h}\times V_h$.
		From the definitions~\eqref{eq:jumpseminorm} and~\eqref{eq:AJ} of $\abs{\cdot}_J$ and $\calA_J$, and the Cauchy--Schwarz inequality, we have
		$
		\calA_J(v,w)
		\leq\abs{v}_J\abs{w}_J.
		$
		Using integration by parts on the volume term of $\calA_h$ in~\eqref{eq:Ah}, the ``DG magic formula''~\eqref{eq:DG_magic}, and $\jmp{(a_1,a_2)^\top}=\jmp{a_1}_x+\jmp{a_2}_y$ with the notation~\eqref{eq:jumpxy}, the bilinear form $\calA_h$ can be rewritten as
		\begin{align*}
			\calA_h(v,w)=&
			-\int_{\calT_h} \W \nabla v \cdot \nabla  (\calM w)+
			\int_{\FhI}  \mvl{ \W \nabla v} \cdot \jmp{\calM w} 
			+\int_{\FhD\cup\Fhtwo}  \W \nabla v\cdot \bn \,\calM w\\=&
			\int_{\calT_h} \text{div}(\W\nabla v)  \calM w
			-\int_{\partial \calT_h} \W\nabla v \cdot \bn_T\, \calM w+
			\int_{\FhI}  \mvl{ \W\nabla  v } \cdot \jmp{\calM w} 
			+\int_{\FhD\cup\Fhtwo}  \W \nabla v\cdot \bn \, \calM w\\
			=&\int_{\calT_h} \calL v \calM w
			-\int_{\FhI} \jmp{\W\nabla v}\mvl{\calM w}
			\\
			=&\int_{\calT_h} \calL v \calM w
			-\int_{\FhI} \big(\jmp{Kv_x}_x+\jmp{v_y}_y\big)\mvl{\calM w}.
		\end{align*}
		The first term is immediately bounded as follows:
		$$
		\int_{\calT_h} \calL v \calM w
		=\int_{\calT_h} \calL v(b w_x+c w_y)
		\le \sqrt2\,\beta\N{\calL v}_{L^2\Th}\N{\nabla w}_{L^2\Th}
		\le \sqrt2\,\beta\,\delta^{-\frac12}\N{\calL v}_{L^2\Th}\vertiii{w}.
		$$
		
		Let $F\in\FhI$ an internal facet and let $T^+, T^-\in \calT_h$ two distinct elements such that $F=\partial T^+\cap \partial T^-$. 
		Using the Cauchy--Schwarz inequality, the inequalities 
		$(a_1+a_2)^2\leq 2(a_1^2+a_2^2)$ and $\jmp{v_{x}}_x^2\leq |\jmp{v_{x}}|^2$ ($n_x^2\leq 1$), we obtain
		\begin{align*}
			&	\abs{-\int_{F}(K\jmp{v_x}_x+\jmp{v_y}_y) (b\mvl{w_x}+c\mvl{w_y}) } \leq
			\beta 	\left(	\int_{F}  \Big(|\jmp{v_x}_x| +|\jmp{v_y}_y|\Big)^2\right)^\frac12   
			\left(	\int_{F}  \Big(|\mvl{w_x}| +|\mvl{w_y}|\Big)^2\right)^\frac12   \\
			&\leq\sqrt{2}\beta
			\left(\int_{F}\frac{\gamma_2 p^2}{h_F}\Big(|\jmp{v_x}|^2 +|\jmp{v_y}|^2\Big)\right)^\frac12   	\left(	\int_{F} \frac{h_F}{\gamma_2 p^2}  \Big((w_x^+)^2+ (w_x^-)^2 +(w_y^+)^2+(w_y^-)^2\Big)\right)^\frac12.
		\end{align*}
		Summing over all internal facets and applying the Cauchy–Schwarz inequality again,
		\begin{align*}
			\abs{-\int_{\FhI}\jmp{\W\nabla v}\mvl{\calM w}}&\leq   	\sqrt{2}\beta
			\abs{v}_J 	\left(	\int_{\FhI} \frac{h_F}{\gamma_2 p^2}  \Big((w_x^+)^2+ (w_x^-)^2 +(w_y^+)^2+(w_y^-)^2\Big)\right)^\frac12.
		\end{align*}
		For the last integral, since $h_F\leq h_T$ for all $F\in \calF_T$, the trace inequality~\eqref{eq:discretetraceinequality} gives
		\begin{align*}
			&\int_{\FhI} \frac{h_F}{\gamma_2 p^2}  \Big((w_x^+)^2+ (w_x^-)^2 +(w_y^+)^2+(w_y^-)^2\Big)\\
			&\leq \sum_{T\in\calT_h} \sum_{F\in\calF_T} \frac{h_F}{\gamma_2p^2} \int_{F} (w_x^2+w_y^2)
			\leq \sum_{T\in\calT_h}  \int_{\partial T} \frac{h_T}{\gamma_2 p^2} (w_x^2+w_y^2)
			\leq \sum_{T\in\calT_h}    \frac{\Ctr^2}{\gamma_2}\int_{T}  (w_x^2+w_y^2)
			\le \frac{\Ctr^2}{\gamma_2\delta}\vertiii{w}^2.
		\end{align*}
		Putting together the bounds above, we obtain the assertion:
		\begin{align*}
			\calA_h(v,w)+	\calA_J(v,w)&\leq \sqrt{2}
			\beta\delta^{-\frac12}	\N{\calL v}_{L^2(\calT_h)} \vertiii{w}+ \nonumber
			\sqrt 2 \beta \Ctr\gamma_2^{-\frac12} \delta^{-\frac12}
			\abs{v}_J  \vertiii{w}+
			\abs{v}_J\abs{w}_J\\&\le
			\big(\sqrt{2}\beta\delta^{-\frac12}(1+
			\Ctr\gamma_2^{-\frac12})+1\big)\vertiii{v}_{\calL} \vertiii{w}.
		\end{align*}
	\end{proof}
	
	\begin{remark}[Least-squares variant]\label{rem:LS}
	We consider also a least-squares variant of the DG formulation~\eqref{eq:variational}, obtained by adding a least-squares stabilization term weighted by a coefficient $\gamma_4>0$. 
	Specifically, we consider the problem
	$$
	\text{Find } u_h \in V_h \text{ such that }
	\calA_h(u_h,v_h)+\calA_J(u_h,v_h)+\gamma_4\int_{\calT_h}\calL u_h\calL v_h
	=L_h(v_h)+\gamma_4\int_{\calT_h} f \calL v_h \quad \forall v_h \in V_h.
	$$
	This is consistent for solutions $u\in H^2(\Omega)$ of the problem~\eqref{eq:Frankl}--\eqref{eq:bndc:Dir}.
	We define the least-squares energy norm for all $v\in V_{*h}$ as
	\begin{align*}
		\vertiii{v}^2_{\mathrm{ls}}:=\int_{\calT_h} \delta(   v_x^2 + v_y^2)+\abs{v}^2_{J}+\gamma_4\int_{\calT_h} \left(  \calL v\right)^2.
	\end{align*}
	The analysis follows similarly to the previous case.
	Coercivity and continuity holds in the $\vertiii{\cdot}_{\mathrm{ls}}$ norm with the same constants as before.
	With this formulation, coercivity and continuity are both established in the same norm $\vertiii{\cdot}_{\mathrm{ls}}$, eliminating the norm mismatch between Propositions~\ref{prop:coercivity} and~\ref{prop:continuity}.
	In particular, quasi-optimality holds with the same norm $\vertiii{\cdot}_{\mathrm{ls}}$ at left- and right-hand side (compare against Corollary~\ref{cor:errorestimate}).
	However, this can lead to the convergence rates that could be expected from the $H^2\Th$ norm of the error, even if the second derivatives of the error are not controlled.
    \end{remark}

	\section{Error analysis} \label{s:erroranalysis}
	We use the notation $A\lesssim B$ to indicate that there exists a constant $C>0$ independent of both the mesh size $h$ and the polynomial degree $p$ such that $A\le C B$. 
	
	Since consistency, discrete coercivity and  boundedness are satisfied, Theorem~1.35 in~\cite{di2011mathematical} applies to the DG method~\eqref{eq:variational} for any polynomial discrete space $V_h\subset \mathbb{P}^{p}(\calT_h)$ and gives the following corollary.
	
	\begin{corollary}[Quasi-optimality]
		\label{cor:errorestimate}
		Under the assumptions on the boundary value problem and the mesh made in Sections~\ref{s:introduction} and~\ref{s:DG}, let $u\in H^2(\Omega)$ solve~\eqref{eq:Frankl}--\eqref{eq:bndc:Dir} and let $u_h$ solve~\eqref{eq:variational} with the penalty parameters $\gamma_1$, $\gamma_2$ and $\gamma_3$ as in Proposition~\ref{prop:coercivity}. 
		Then, the following quasi-optimality error estimate holds true:
		\begin{equation}\label{eq:quasioptimality}
			\vertiii{u-u_h}\leq \left(1+4M\right) \inf_{v_h\in V_h}\vertiii{u-v_h}_{\calL},
		\end{equation}
		with 
		$M$ as in Proposition~\ref{prop:continuity}.
	\end{corollary}

We discuss three different choices for the discrete space $V_h$: the standard piecewise polynomial space $\IP^p(\calT_h)$ in Section~\ref{s:standardpoly}, a quasi-Trefftz subspace in Section~\ref{s:quasiTrefftz}, and an embedded Trefftz subspace in Section~\ref{s:embeddedTrefftz}.

\subsection{Standard polynomials}\label{s:standardpoly}
We first consider as discrete space $V_h$ the piecewise polynomial space $V_h:=\IP^p(\calT_h)$.
Given a mesh $\calT_h$, we define a \emph{covering} $\Tc = \{\mathfrak{T}\}$ of $\calT_h$ a set of shape-regular open triangles $\mathfrak{T}$ such that, for each $T\in \calT_h$, there exists $\mathfrak{T}\in \Tc$ with $T\subset \mathfrak{T}$ (\cite[Def.~4.27]{cangiani2022hp}).
	
	\begin{assumption}[Covering of~$\calT_h$, {\cite[Ass.~18]{CangianiDGH2017}, \cite[Ass.~4.28]{cangiani2022hp}}]
		\label{ass:covering}
		There exists a positive integer~$N_{\Omega}$ independent of the mesh parameters, such that, for any mesh~$\calT_h \in \{\calT_h\}_{h > 0}$, there exists a covering~$\Tc$ of~$\calT_h$ satisfying
		\begin{equation*}
			\mathrm{card} \big\{T' \in \calT_h \, : \, T' \cap \mathfrak{T} \neq \emptyset \text{ for some } \mathfrak{T} \in \Tc \text{ with } T \subset \mathfrak{T}\big\} 
			\le N_{\Omega} \qquad \forall T \in \calT_h.
		\end{equation*}
		Moreover,
		$h_{\mathfrak{T}} := \mathrm{diam}(\mathfrak{T}) \lesssim h_T$ for each pair~$T\in \calT_h$ and~$\mathfrak{T} \in \Tc$ with~$T \subset \mathfrak{T}$.
	\end{assumption}	
If $\Omega$ is convex, all mesh elements $T$ with $\calF_T\subset\FhI$ are triangles, and all remaining elements have two internal and one boundary facet, then Assumption~\ref{ass:covering} holds with $N_\Omega=1$ and $\mathfrak T=T$ for all internal elements (recall that the star-shaped property~\ref{it:starshaped} implies shape-regularity).

Given a Lipschitz domain $D\subset \IR^2$ and $s\in\IN$, the Stein's operator $\mathfrak{E}_{D}:H^s(D)\to H^s(\IR^2)$ is a linear extension operator (\cite[Thm.~5 in Ch.~VI]{stein1970singular}) such that
	\begin{equation}\label{eq:Stein}
		(\mathfrak{E}_{ D} v)_{|_D} = v \quad \text{ and }\quad  \Norm{\mathfrak{E}_{ D} v}_{H^s(\IR^2)}\leq C_{\mathfrak{E}_{ D}}\Norm{v}_{H^s(D)},  \qquad \forall  v\in H^s(D),
	\end{equation}
	where the constant $C_{\mathfrak{E}_{D}}>0$ depends only on $s$ and the shape of $D$.
	We now recall the following approximation result from~\cite[Lemma 4.31]{cangiani2022hp}.
	Since~\cite[Lemma 4.31]{cangiani2022hp} is based on~\cite{babuvska1987optimal}, in the next two result we admit non-integer Sobolev exponent $s$.

\begin{lemma}\label{lem:estimatePihp}
Let $\calT_h \in \{\calT_h\}_{h > 0}$ and~$\Tc$ be its corresponding covering from Assumption~\ref{ass:covering}.
For any $T\in \calT_h$, $\mathfrak T\in\Tc$ with $T\subset\mathfrak T$ and $v\in H^s(T)$
for some $s\geq 0$, there exists $\Pi_{h,p} v
\in \IP^p(T)$, such that
\begin{align}\label{eq:Pihp_T}
\Norm{v-\Pi_{h,p}v}_{H^q(T)}\le& \Capp \frac{h_T^{l-q}}{p^{s-q}} \Norm{\mathfrak{E}_T v}_{H^s(\mathfrak{T})}, \qquad  \forall 0\leq q\leq s,
\end{align}
where $l:=\min\{s,p+1\}$ and $\Capp>0$ is a constant, independent of $h$ and $p$.
\end{lemma}
	
\begin{theorem}[DG convergence rates]\label{thm:finalerror}
Given~$p\in\IN$, $p\ge 2$ and $s\ge \frac52$, let~$u \in H^2(\Omega) \cap H^{s}(\calT_h)$ be the solution to~\eqref{eq:Frankl}--\eqref{eq:bndc:Dir} and~$u_h\in \IP^p(\calT_h) $ be the solution to the DG method~\eqref{eq:variational} with~$V_h=\IP^p(\calT_h)$.
Under the mesh assumptions~\ref{it:starshaped}--\ref{it:graded} (with star-shaped and grading parameters $r_\star$ and $C_g$) and choosing $\gamma_1$, $\gamma_2$ and $\gamma_3$ as in Proposition~\ref{prop:coercivity}, the following convergence rate holds
\begin{align}\nonumber
\vertiii{u-u_h}^2
\le &\big(1+4M\big)^2
\sum_{T\in \calT_h}\frac{h_T^{2(l-2)}}{p^{2s-5}}\Norm{\mathfrak{E}_T u}_{H^s(\mathfrak{T})}^2 \Capp
\Big[\|K^2\|_{L^{\infty}(T)}+1+6\frac{C_g}{r_\star}(C_g^2\gamma_1p^{-4}+2\gamma_2+\gamma_3)\Big]
\\
\lesssim& \sum_{T\in \calT_h} \frac{h_T^{2(l-2)}}{p^{2(s-\frac52)}}\Norm{\mathfrak{E}_T u}_{H^s(\mathfrak{T})}^2,
\label{eq:ERROR-PP}
\end{align}
where $l:=\min\{s,p+1\}$, with $M$ as in Proposition~\ref{prop:continuity}.
\end{theorem}
\begin{proof}
Let~$v\in\IP^p(\calT_h)$ be defined as~$v_{|_T} = \Pi_{h,p}(u_{|_T})$ for all~$T \in \calT_h$ and $\Pi_{h,p}$ as in Lemma~\ref{lem:estimatePihp}.
We observe that $|\jmp{v}|^2=(v^+-v^-)^2\leq 2((v^+)^2+(v^-)^2)$ on internal facets $F=\partial T^+\cap \partial T^-$, and $|\jmp{v}|^2=v^2$ on a boundary facets. 
		This allows the sum over facets in $\abs{u-v}_J^2$ to be rewritten as a sum over elements.
		We use the trace estimate on star-shaped domains from~\cite[Lemma~2]{MoPe18},
			weighted with $a=p$:
			$$\|w\|_{L^2(\partial T)}^2
			\le\frac{p+2}{r_\star h_T}\|w\|_{L^2(T)}^2+\frac{h_T}{pr_\star}\|\nabla w\|_{L^2(T)}^2 \qquad \forall w\in H^1(T), \quad T\in\calT_h.$$
		Applying the approximation estimate~\eqref{eq:Pihp_T}, using $h_F^{-1}\le C_gh_T^{-1}$ from~\eqref{eq:gradedmesh} and $p\ge2$,
		we bound the term on the right-hand side of the quasi-optimality~\eqref{eq:quasioptimality} as follows:
		\begin{align*}
			\vertiii{u-v}^2_{\calL}=&\N{\calL (u-v)}^2_{L^2 \Th}+\abs{u-v}^2_{J}
			\\
			\le &\sum_{T\in \calT_h}\Bigg[2\|K^2\|_{L^{\infty}(T)} \N{(u-v)_{xx}}^2_{L^2(T)}
			+2\N{ (u-v)_{yy}}^2_{L^2(T)}
			+2C_g^3\frac{\gamma_1}{h_T^3}\N{u-v}^2_{L^2(\partial T)}\\&
			+2C_g\frac{\gamma_2p^2}{h_T}\N{(u-v)_x}^2_{L^2(\partial T)}+2C_g\frac{\gamma_2p^2}{h_T}\N{(u-v)_y}^2_{L^2(\partial T)} + 2 C_g \frac{\gamma_3 p^2}{h_T} \N{(u-v)_t}^2_{L^2(\partial T)}
			\Bigg]\\
			\le&\sum_{T\in \calT_h}\!\Bigg[2\|K^2\|_{L^{\infty}(T)} \N{(u-v)_{xx}}^2_{L^2(T)}
			+2\N{ (u-v)_{yy}}^2_{L^2(T)}+ 2\frac{C_g}{r_\star}
			\bigg(C_g^2\frac{(p+2)\gamma_1}{h_T^4}\N{u-v}_{L^2(T)}^2\\&
			+\Big(C_g^2\frac{\gamma_1}{ph_T^2}
			+\frac{p^2(p+2)(2\gamma_2+\gamma_3)}{h_T^2}
			\Big)\abs{u-v}_{H^1(T)}^2
			+p(2\gamma_2+\gamma_3)\abs{u-v}_{H^2(T)}^2\bigg)\Bigg]
			\\
			\le &
			\sum_{T\in \calT_h}\frac{h_T^{2(l-2)}}{p^{2s-5}}\Norm{\mathfrak{E}_T u}_{H^s(\mathfrak{T})}^2 \Capp
				\Bigg[2\big(\|K^2\|_{L^{\infty}(T)}+1\big)p^{-1}
				\\&+ 2\frac{C_g}{r_\star}
				\bigg(	C_g^2	\gamma_1 (p^{-4}+2p^{-5})
				+\Big(C_g^2 \gamma_1 p^{-4}+2\gamma_2+4\gamma_2p^{-1}+\gamma_3+2\gamma_3p^{-1}\Big)
				+(2\gamma_2+\gamma_3)
				\bigg)\Bigg]
            \\
            \le&
            \sum_{T\in \calT_h}\frac{h_T^{2(l-2)}}{p^{2s-5}}\Norm{\mathfrak{E}_T u}_{H^s(\mathfrak{T})}^2 \Capp
			\Big[\|K^2\|_{L^{\infty}(T)}+1+6\frac{C_g}{r_\star}(C_g^2\gamma_1 p^{-4}+2\gamma_2+\gamma_3)\Big].  
		\end{align*}
		Therefore, combining the above estimate with the quasi-optimality~\eqref{eq:quasioptimality}, we obtain~\eqref{eq:ERROR-PP}.
	\end{proof}

\begin{corollary}
	Let the hypotheses of Theorem~\ref{thm:finalerror} hold. If the continuous solution $u\in H^s(\Omega)$ then the following bound holds for $l:=\min\{s,p+1\}$:
	\begin{equation*}
		\vertiii{u-u_h}\lesssim\frac{h^{l-2}}{p^{s-\frac52}}\Norm{ u}_{H^s(\Omega)}.
	\end{equation*} 
\end{corollary}
	
\begin{remark}[Suboptimality of $H^1(\Omega)$ convergence rates]\label{rem:suboptimal}
	The bound~\eqref{eq:ERROR-PP} immediately allows to control the $H^1\Th$ seminorm of the Galerkin error with the same bound, up to a factor $\delta^{-1}$.
	However, the $H^1\Th$ convergence rates are $h$-suboptimal by $h^{-1}$ and $p$-suboptimal by $p^\frac32$.
	Thus, the $h$-convergence of the DG method is guaranteed for $p\ge 2$ and $s>2$, and the $p$-convergence of the DG method is guaranteed for $s>\frac52$.
	This suboptimality originates from a mismatch between the norms in which
	coercivity and continuity of $\mathcal{A}_h+\mathcal{A}_J$ are established:
	coercivity holds in the energy norm $\vertiii{\cdot}$, which only controls
	broken first-order derivatives, whereas continuity requires the stronger
	residual norm $\vertiii{\cdot}_{\mathcal L}$, which scales like a broken
	$H^2$-norm. 
	The suboptimality in $h$ is consistent with the result obtained for the finite element methods in~\cite[Thm.~4.2]{aziz1980finite}, \cite[Thm.~ 4.1]{sermer1983galerkin}.
	This suggests that the loss is a consequence of the formulation based on a first-order multiplier rather than of the DG discretization.
	This also agrees with the numerical experiments of
	Section~\ref{s:numericalexperiments} at least for some values of $p$, where the $H^1\Th$ convergence rates match the
	theoretical prediction $\mathcal O(h^{p-1})$.
\end{remark}
\begin{remark}[Stabilization terms]
The stabilization term $\frac{\gamma_3 p^2}{h_F} \int_{\FhD} u_t v_t$ in~\eqref{eq:AJ} is not necessary if one is interested in \emph{a priori} error bounds with respect to $h$ only.
We introduce this term to derive $hp$-explicit bounds in the case of meshes with curved edges.
For meshes with straight facets $F$, one can exploit the following inverse inequality from~\cite[Thm. 3.91]{Schwab98}:
\begin{align*} 
\N{v_t}_{L^2(F)}&\leq 2\sqrt3 
\, p^2\, h_F^{-1} \N{v}_{L^2(F)}.
\end{align*}		
Using this estimate, one can control $\Norm{v_t}_{L^2(F)}$ in~\eqref{eq:vnvtQnt} with the stabilization term $\frac{\gamma_1 p^6}{h_F^3} \int_{\FhI \cup \FhD} \jmp{v}^2$.
However, the extension of the above  $p$-explicit inverse estimate to curved facets is not straightforward.
To maintain $p$-explicit control over tangential derivatives on curved facets, we therefore introduce the stabilization term $\frac{\gamma_3 p^2}{h_F} \int_{\FhD} u_t v_t$ which allows us to avoid relying on such inverse estimates.
Finally, we note that the term involving the penalty $\gamma_1$, although not used to control tangential derivatives on facets, is necessary to ensure that $\vertiii{\cdot}$ defines a norm.
\end{remark}

\subsection{Quasi-Trefftz polynomials} \label{s:quasiTrefftz}
For each element~$T \in \calT_ h$ and a chosen point $\bx_T \in T$, we define the \emph{polynomial quasi-Trefftz space} of degree $p\in \IN$ associated with the  equation~$\calL u = f$ in $T$ as
\begin{equation}\label{eq:QT}
\qT^p_f(T):=\big\{ v\in \IP^{p}(T) \, \mid \, D^{\mi} \calL v (\bx_T)=D^{\mi} f(\bx_T) \quad \forall \mi\in \IN^{2}, 
\;|\mi|\leq p-2 \big\},
\end{equation}
where $D^{\mi}:= \partial_{x}^{i_1} \partial_{y}^{i_2} $ denotes the partial derivative corresponding to the multi-index~$\mi=(i_1,i_2)$.
The quasi-Trefftz space~\eqref{eq:QT} is an affine space and can be written as $\qT^p_f(T)=\qT^p_0(T)+u_f^T$, where $\qT^p_0(T)$ is the linear space associated to the homogeneous equation $\calL u=0$ and $u_f^T$ is a particular approximate solution.
This construction allows the quasi-Trefftz method to handle non-homogeneous source term~$f$ by first constructing an element-wise approximate particular solution~$u_{h,f}$ (with $(u_{h,f})_{|_T}:=u_f^T$), and then computing the solution of a homogeneous problem,
see~\cite[\S5]{10.1093/imanum/drae094} and \cite{imbertgerard2025localtaylorbasedpolynomialquasitrefftz} for more details.
A basis for the quasi-Trefftz space $\qT^p_0(T)$ can be constructed using the recursive procedure described in~\cite[\S2.4]{10.1093/imanum/drae094}, which computes the coefficients of the monomial expansion of each basis function explicitly. 
This construction relies on a non-degeneracy assumption on the differential operator $\calL$, stated in~\cite[eq.~(9)]{10.1093/imanum/drae094}, which requires that at least one coefficient of a pure highest-order derivatives is nonzero at the expansion point $\bx_T$. 
In the case of the Frankl operator~\eqref{eq:PDE}, this condition is satisfied since the coefficient of $u_{yy}$ is equal to $1$. 
The dimension of the quasi-Trefftz space~\eqref{eq:QT} is $\dim(\qT^p_f(T))=2p+1=\mathcal O(p)$,
whereas the dimension of the standard polynomial space is $\dim(\IP^p(T))=\frac12(p+1)(p+2)=\mathcal O(p^2)$, see~\cite[\S2.4]{10.1093/imanum/drae094}.
This leads to a significant reduction of the total number of degrees of freedom.

The key approximation property of the space $\qT^p_f(T)$ is that the Taylor polynomial of order~$p + 1$ (and degree~$p$) centered at~$\bx_T$ of the exact solution $u$, denoted by $T^{p+1}_{\bx_T}[u]$, belongs to $\qT^p_f(T)$ (see~\cite[Thm.~2.4]{10.1093/imanum/drae094}). 
This ensures that the quasi-Trefftz space approximates with high order in $h$ the smooth PDE solutions:
under the star-shaped property assumption~\ref{it:starshaped}, if $u \in C^{p + 1}(T)$ 
solves $\calL u=f$ on
$T \in \calT_h$, then (see again~\cite[Thm.2.4]{10.1093/imanum/drae094})
\begin{equation}\label{eq:ERROR-TAYLOR}
\inf_{v\in\qT^p_f(T)}|u-v|_{C^q(T)} \leq 	|u-T_{\bx_T}^{p+1}[u]|_{C^q(T)} 
\le \frac{2^{p+1-q}}{(p+1-q)!}h_T^{p+1-q} |u|_{C^{p+1}(T)} \quad \forall q\in\IN,\, q\leq p.
\end{equation}
The global polynomial quasi-Trefftz space is 
$$\qT^p_f(\calT_h):=\{v\in L^2(\Omega)\;|\; v_{|_T} \in\qT^p_f(T) \;\; \forall\, T\in \calT_h\}.$$

\begin{theorem}[Quasi-Trefftz DG convergence rates]\label{thm:errorquasiTrefftz}
Given~$p\in\IN$, $p\ge2$, let~$u \in H^2(\Omega)$ be the solution to~\eqref{eq:Frankl}--\eqref{eq:bndc:Dir} with $u|_T\in C^{p+1}(T)\cap C^1(\overline T)$ for each $T\in\calT_h$.
Let $u_h\in \qT^p_f(\calT_h)$ be the solution to the DG method~\eqref{eq:variational} with~$V_h=\qT^p_0(\calT_h)$.
Under the mesh assumptions~\ref{it:starshaped}--\ref{it:graded} and choosing $\gamma_1$, $\gamma_2$ and $\gamma_3$ as in Proposition~\ref{prop:coercivity}, the following error bound holds
\begin{align}\nonumber
\vertiii{u - u_h}^2 \le&
\big(1+4M\big)^2
\sum_{T\in \calT_h} h_T^{2(p-1)} |T||u|_{C^{p+1}(T)}^2\frac{2^{2p}}{\big((p-1)!\big)^2}\Bigg[
\frac12\|K^2\|_{L^{\infty}(T)}+\frac12
+ C_g^3\frac{\gamma_1}{r_\star}
+ 4C_g \frac{(2\gamma_2+\gamma_3)}{r_\star}
\Bigg]
\\
\lesssim& \sum_{T\in \calT_h}  h_T^{2(p-1)} |T|\abs{u}_{C^{p+1}(T)}^2,
\label{eq:ERROR-QT}
\end{align}
with $M$ as in Proposition~\ref{prop:continuity}.
If $u\in C^{p+1}(\Omega)$, then $\vertiii{u - u_h} \lesssim h^{p-1}  |u|_{C^{p+1}(\Omega)} |\Omega|^\half $.
\end{theorem}
\begin{proof}
Let~$v\in\qT^p_f(\calT_h)$ be defined as~$v_{|_T} = T^{p+1}_{\bx_T}[u_{|_{T}}]$ for all~$T \in \calT_h$.
For each element $T\in\calT_h$, we have 
$|\partial T|\le \frac{2|T|}{r_\star h_T}$ by~\cite[Lemma~4.1]{10.1093/imanum/drae094}, 
and the bounds 
$\|z\|_{L^2(T)}^2\le |T|\|z\|_{C^0(T)}^2$,
$\|z\|_{L^2(\partial T)}^2\le |\partial T|\|z\|_{C^0(T)}^2$
for all $z\in C^0(\overline{T})$.
Then, applying the approximation result for~$T^{p+1}_{\bx_T}$ in~\eqref{eq:ERROR-TAYLOR} to the first inequality in the proof of Theorem~\ref{thm:finalerror}, we have
\begin{align*}
\vertiii{u-v}^2_{\calL}
\le &\sum_{T\in \calT_h}\Bigg[\big(2\|K^2\|_{L^{\infty}(T)} +2\big)|T|
|u-v|^2_{C^2(T)}
\\&\qquad
+4C_g^3\frac{\gamma_1}{h_T^4r_\star}|T|\N{u-v}^2_{C^0(T)}
+4C_g \frac{(2\gamma_2+\gamma_3) p^2}{h_T^2r_\star}|T||u-v|^2_{C^1(T)}
\Bigg]
\\
\le&\sum_{T\in \calT_h} h_T^{2p-2}|T| |u|_{C^{p+1}(T)}^2\frac{2^{2p}}{\big((p-1)!\big)^2}\Bigg[
\frac12\|K^2\|_{L^{\infty}(T)} +\frac12
+\frac{16}{p^2(p+1)^2}C_g^3\frac{\gamma_1}{r_\star}
+4C_g \frac{(2\gamma_2+\gamma_3)}{r_\star}
\Bigg].
\end{align*}
We recall that 
$\frac4{p(p+1)}\le1$ and $\frac{2^p}{(p-1)!}\le 4$ for all $p\ge2$.
\end{proof}

\subsection{Embedded Trefftz polynomials\label{s:embeddedTrefftz}}

	For each element~$T \in \calT_ h$, we define the local \emph{embedded Trefftz space} of degree $p\in \IN$, $p\ge2$, associated with the equation~$\calL u = f$ in $T$ as
	\begin{equation}\label{def::ET}
	\eT^p_f(T) :=\big\{v\in \Pp{p}{T}\, \mid \,  \Pi^{p-2} (\calL v) = \Pi^{p-2} f \text{ in } T\big\},
	\end{equation}
	where~$\Pi^{p-2}$ is the~$L^2(T)$-orthogonal projection operator onto the space~$\Pp{p - 2}{T}$. 
	The embedded Trefftz space~\eqref{def::ET} is an affine space and can be written as $\eT^p_f(T)=\eT^p_0(T)+u_f^T$, where $\eT^p_0(T)$ is the linear space associated to the homogeneous equation $\calL u=0$ and $u_f^T $ is a particular element-wise approximate solution.
	The global embedded Trefftz space is 
	$\eT^p_f(\calT_h):=\{v\in L^2(\Omega)\;|\; v_{|_T} \in\eT^p_f(T) \;\; \forall\, T\in \calT_h\}.$

	Rather than constructing Trefftz basis functions explicitly, the embedded Trefftz method, introduced in~\cite{lehrenfeld2023embedded,lozinski19}, enforces the Trefftz property in a weak sense by embedding the Trefftz space into the standard polynomial space $\IP^p(T)$. 
	Unlike the quasi-Trefftz approach, this procedure does not require Taylor expansions of the PDE coefficients or the source term. Instead, a small element-wise singular value decomposition is computed, which provides both a basis for the linear space $\eT^p_0(T)$ and a particular solution using the associated pseudoinverse.
	
	In the definition~\eqref{def::ET} of $\eT^p_f(T)$, it is possible to take projections other than the $L^2(T)$-orthogonal one;
	the approximation properties of the embedded Trefftz space depend on this choice.
	For the $L^2(T)$ projection adopted here, the numerical experiments in Section~\ref{s:numericalexperiments} show the same convergence rates in $h$ as with the full polynomial spaces $\IP^p\Th$.
	Moreover, the dimension of the space coincides with that of the quasi-Trefftz space, i.e.\ $\dim(\eT^p_f(T))=2p+1=\mathcal O(p)$,
	which leads to a significant reduction in the total number of degrees of freedom.
	A rigorous theoretical analysis of the approximation properties of embedded Trefftz spaces remains challenging and problem-dependent. 
	A recent unifying framework for Trefftz-like methods, including an error analysis for embedded Trefftz discontinuous Galerkin methods applied to some scalar elliptic PDEs, is provided in~\cite{LLSV_ARXIV_2024}.

\section{Numerical experiments} \label{s:numericalexperiments}

We present numerical experiments that validate the theoretical results and show additional properties of the DG method. 
We compare the three discrete spaces introduced in Section~\ref{s:erroranalysis}: the standard, quasi-Trefftz, and embedded Trefftz polynomial spaces.
The proposed DG method has been implemented using \texttt{NGSolve}~\cite{ngsolve} and \texttt{NGSTrefftz}~\cite{ngstrefftz}\footnote{Replication data are available in \cite{perinati_2026_18998989}.}.
We employ unstructured triangular meshes with curved boundary elements obtained through an isoparametric mapping of polynomial degree consistent with that of the discrete space used. 
Differently from the theoretical setting, the standard and embedded Trefftz polynomial bases are defined on the reference element and mapped to the physical elements, whereas the quasi-Trefftz bases are constructed directly on the physical elements, as in Section~\ref{s:quasiTrefftz}.
We also tested the code using unmapped standard and embedded Trefftz polynomial spaces and observed similar results.
Unless stated otherwise, the stabilization parameters are chosen as $\gamma_1=10$ and $\gamma_2=\gamma_3=0.1$ for all the respective facets.

We consider the case $K(y)=y$ of the Tricomi equation.
We choose the computational domain $\Omega$ shown in Figure~\ref{fig:domainexperiments}.
The elliptic boundary $\Gamma_0$ is defined as the union of the two segments
\begin{equation}\label{eq:Gamma0experiments}
	\Gamma_0=\big\{(x,d-d|x|)\mid |x|\le1\big\},
\end{equation}
for some $d>0$,
while the hyperbolic boundary consists of the union of the two 
characteristic curves~\eqref{eq:chareq}, which are explicitly given by
\begin{equation}\label{eq:chartricomi}
	x = -1 + \frac{2}{3}(-y)^{3/2} \quad \text{on } \Gamma_1,
	\qquad
	x = 1 - \frac{2}{3}(-y)^{3/2} \quad \text{on } \Gamma_2,
\end{equation}
and intersect at the point $(0,-\left(\frac32\right)^\frac23)\approx(0,-1.31037)$.
	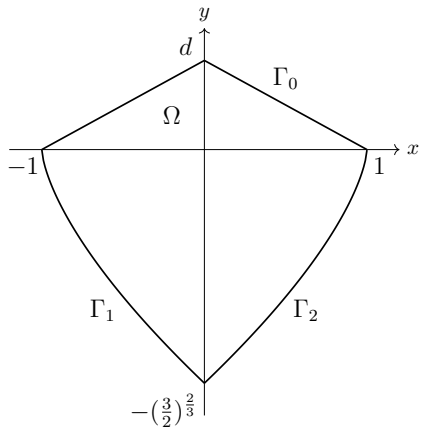
\begin{figure}[htbp]
		\begin{center}
			\begin{tikzpicture}[scale=0.8]
				\begin{axis}[
					axis lines = middle,
					axis line style = {->},
					xlabel = {$x$},
					ylabel = {$y$},
					xlabel style = {right},
					ylabel style = {above},
					xtick = \empty,
					ytick = \empty,
					enlargelimits = true,
					domain=-1:1,
					samples=100,
					clip=false,
					width=8cm,
					height=8cm
					]
					\addplot [domain=-1:0, color=black, thick] {-(3/2*(x+1))^(2/3)};
					\addplot [domain=0:1, color=black, thick] {-(-3/2*(x-1))^(2/3)};
					\addplot[
					color=black,
					thick,
					smooth
					] coordinates {
						(1, 0)
						(0, 0.5)
					};
					\addplot[
					color=black,
					thick,
					smooth
					] coordinates {
						(0, 0.5)
						(-1, 0)
					}; 
					\node at (axis cs: -0.2,0.1) [anchor=south] {\large $\Omega$};
					\node at (axis cs: -1.05,0) [anchor=south,xshift=-5pt,yshift=-16pt] {\large $-1$};
					\node at (axis cs: -0.05,0.68) [anchor=south,xshift=-5pt,yshift=-16pt] {\large $d$};
					\node at (axis cs: 1.05,0) [anchor=south,xshift=2pt,yshift=-15pt] {\large $1$};
					\node at (axis cs: -0.5 ,-0.9) [anchor=east] {\large $\Gamma_1$};
					\node at (axis cs: 0.5,-0.9) [anchor=west] {\large $\Gamma_2$};
					\node at (axis cs: 0.5,0.3) [anchor=south] {\large $\Gamma_0$};
					\node at (axis cs: -0.25,-1.6) [anchor=south] {\large $-(\frac32)^\frac23$};
				\end{axis}
			\end{tikzpicture}
		\end{center}
\caption{Computational domain $\Omega$ for to the case $K(y) = y$.}
\label{fig:domainexperiments}
\end{figure}

As a test case, we consider the boundary value problem~\eqref{eq:Frankl}--\eqref{eq:bndc:Dir} with $K(y) = y$ on the domain $\Omega$ shown in Figure~\ref{fig:domainexperiments} with $d=0.5$ (see~\eqref{eq:choicecodemult} below for this choice).
The Dirichlet boundary data $g$
and the right-hand side $f$ are chosen such that the exact solution is
\begin{equation}\label{eq:usolExample1}
	u(x,y)=(1-x)^2(1+x)y^3(1-y)[9(1+x)^2+4y^3].
\end{equation}
This solution coincides with that employed in the numerical example studied in~\cite[p.~480]{sermer1983galerkin}.

In Section~\ref{s:choiceMorawetz} we discuss the choice of the Morawetz multiplier for this particular test case. 
We study the $h$-convergence and the $p$-convergence of the method in Sections~\ref{s:hconv} and~\ref{s:pconv}, respectively.
In Section~\ref{s:choicepenalty} we investigate the sensitivity of the method with respect to the choice of the penalty parameters.

\subsection{Choice of the Morawetz multiplier} \label{s:choiceMorawetz}

As in~\cite[eq.~(2.4)]{aziz1980finite}, we assume that the coefficients $b$ and $c$ of the Morawetz multiplier~\eqref{eq:Multiplier} are affine functions in the form:
\begin{equation}\label{eq:bcAffine}
	b(x) = b_0 + b_1x, \qquad c(y) = c_0 + c_1y,
\end{equation}
where $b_0, b_1, c_0, c_1\in\IR$ are constants to be chosen.
This choice ensures that the regularity condition~\ref{ass:A1} is automatically satisfied.
The conditions~\cite[eq.~(2.5)]{aziz1980finite} on the constants $b_0,b_1,c_0,c_1$ are sufficient to ensure the validity of assumptions~\ref{ass:A2} and~\ref{ass:A3} , which correspond to~\cite[Lemma 2.1 (i), (iii)]{aziz1980finite}, for domains $\Omega$ as in Figure~\ref{fig:domain} and for functions $K$ as in~\eqref{eq:kappa}. 

We now reformulate the conditions~\ref{ass:A2}--\ref{ass:A4} as constraints on the constants $b_0, b_1, c_0, c_1$ for the Tricomi problem on the domain shown in Figure~\ref{fig:domainexperiments} with $\Gamma_0$ as in~\eqref{eq:Gamma0experiments}.

Using the affine expression~\eqref{eq:bcAffine}
of the multiplier, the Tricomi coefficient $K(y)=y$, the parametrisations $x=1-\frac23(-y)^{3/2}$ of $\Gamma_2$~\eqref{eq:chartricomi} and $y=d(1-|x|)$ of $\Gamma_0$~\eqref{eq:Gamma0experiments}, and that the outward normal on $\Gamma_0$ is $(d\,\sign(x),1)^\top\frac{1}{\sqrt{1+d^2}}$, these inequalities can be restated as:
$$
\begin{cases}
-Kb_x+(Kc)_y\ge\delta>0,\\ 
b_x-c_y\ge\delta>0,\\
b+c\sqrt{-K}\le0,\\
\bm\cdot\bn=bn_x+cn_y\ge0,
\end{cases}
\implies
\begin{cases}
(2c_1-b_1)y + c_0>0&\oon\Omega, \;  \ref{ass:A2},\\
b_1-c_1>0 &\oon\Omega,\;  \ref{ass:A2},\\
b_0+b_1\big(1-\frac23(-y)^{3/2}\big)+(c_0+c_1y)\sqrt{-y}\le0 & \oon \Gamma_2,\; \ref{ass:A3},\\
(b_0+b_1x)d\,\sign(x)+c_0+c_1d(1-|x|)\ge0 & \oon \Gamma_0,\; \ref{ass:A4}.
\end{cases}
$$
Making explicit the range of the Cartesian coordinates in $\Omega$, $\Gamma_2$ and $\Gamma_0$, these are equivalent to
\begin{equation}\label{eq:suffconditions}
\begin{cases}
(2c_1-b_1)y + c_0>0&-(3/2)^{2/3}\le y\le d,\\
b_1-c_1>0,\\
b_0+b_1 - (\frac23b_1+c_1)(-y)^{3/2}+c_0\sqrt{-y}\le0 & -(3/2)^{2/3}\le y\le 0,\\
(b_1-c_1)xd+c_0+c_1d\ge|b_0|d & 0\le x\le1.
\end{cases}
\end{equation}
A set of sufficient conditions is:
$$
b_1=2c_1>0,\qquad
c_0>0,\qquad 
2c_1+\Big(\frac32\Big)^{1/3}c_0\le -b_0\le c_1+\frac{c_0}d,
$$
which gives also $\delta=\min\{c_0,c_1\}$.
These sufficient conditions imply $d<(2/3)^{1/3}\approx0.874$.
In all the numerical tests we adopt the choice
\begin{equation}\label{eq:choicecodemult}
b(x)=-2+\frac12 x, \qquad c(y)=1+\frac14 y, \qquad d=0.5,
\end{equation}
which satisfies all requirements and gives $\delta=\frac14$ in condition~\ref{ass:A2}.

\subsection{\texorpdfstring{$h$}{h}-convergence}\label{s:hconv}
First, we study the convergence of the DG method under $h$-refinement for fixed polynomial degree $p = 2, 3, 4$ for the Tricomi problem with exact solution $u$ given in~\eqref{eq:usolExample1}.
In Figure~\ref{fig:hconv} we show the errors computed for the three discrete spaces described in Section~\ref{s:erroranalysis} on a sequence of meshes with mesh sizes $h=2^{-2},\dots, 2^{-6}$.
For the quasi-Trefftz and embedded Trefftz spaces, the error measured in the energy norm $\vertiii{\cdot}$ converges with order $\mathcal{O}(h^{p-1})$, in agreement with Theorem~\ref{thm:errorquasiTrefftz}.
In the $L^2(\Omega)$ norm, we observe a convergence rate of $\mathcal{O}(h^{p})$ for even polynomial degrees and of $\mathcal{O}(h^{p-1})$ for odd $p$.
For the standard polynomial space, the convergence rate in the energy norm is at least $\mathcal{O}(h^{p-1})$, with a higher rate of $\mathcal{O}(h^{2})$ observed for $p=2$.
The $L^2(\Omega)$ error converges with rate roughly $\mathcal{O}(h^{p})$ for $p=3, 4$ and $\mathcal{O}(h^{p+1})$ for $p=2$.
Overall, the standard polynomials achieve better accuracy in both norms compared to the Trefftz versions.
\begin{figure}[!htbp]
	\includegraphics[width=\linewidth]{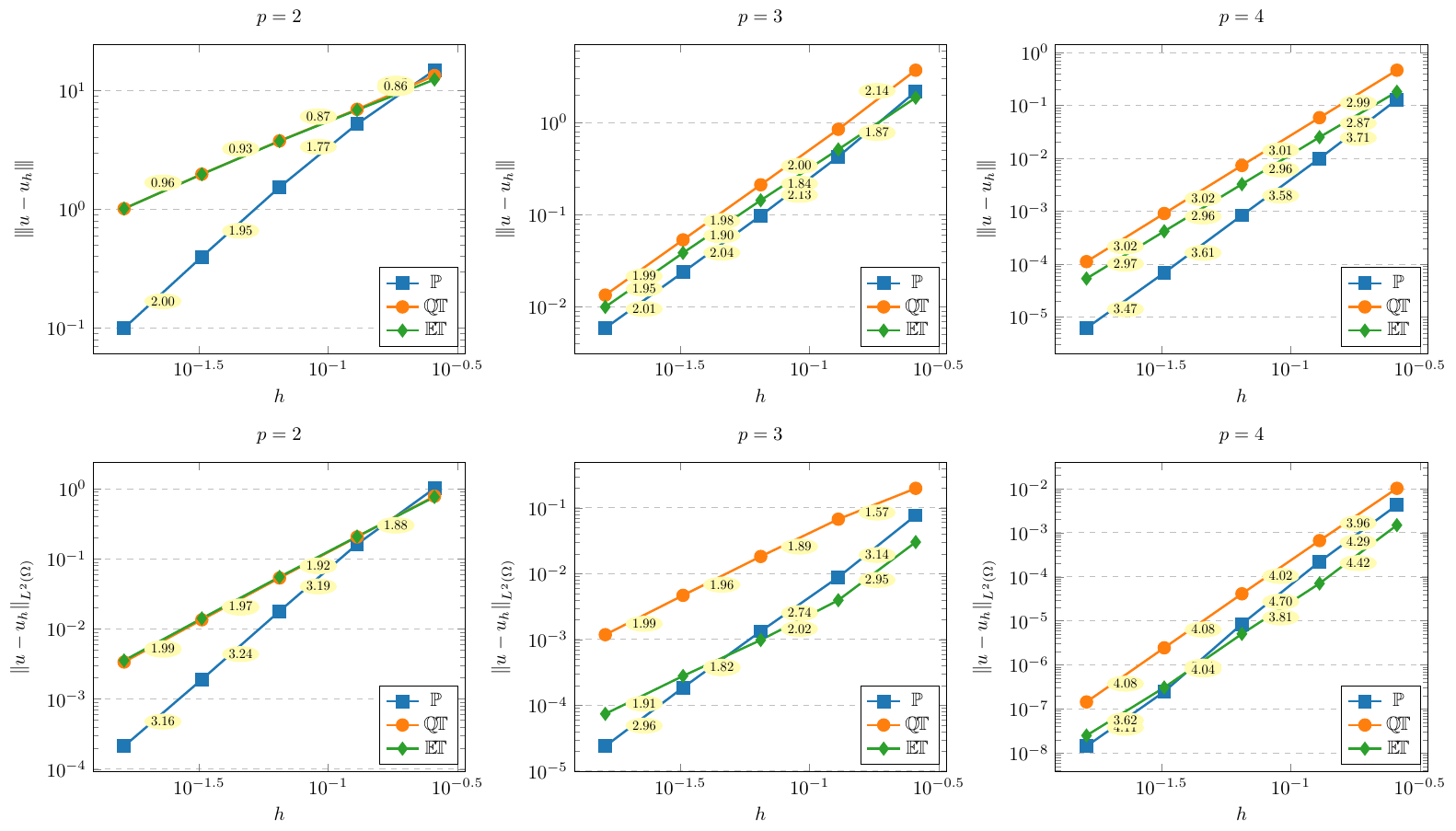}
	\caption{$h$-convergence in the norms $\vertiii{\cdot}$ (first row)  and $\Norm{\cdot}_{L^2(\Omega)}$ (second row) for the Tricomi problem with exact solution $u$ in~\eqref{eq:usolExample1} for the three discrete spaces considered.
	The empirical algebraic convergence rates are shown along each segment.}
	\label{fig:hconv}
\end{figure}

We also consider the least-square variant of the method, introduced in Remark~\ref{rem:LS}, where the term $\gamma_4 \int_{\calT_h}(\calL u_h-f)\calL v_h$ is added to the formulation~\eqref{eq:variational}.
In Figure~\ref{fig:hconv4} we show the $h$-convergence of such variant when the parameter $\gamma_4$ is set to $1$.
Compared to the previous case ($\gamma_4=0$) we observe no significant difference for the quasi-Trefftz and embedded Trefftz methods, whereas the standard DG method exhibits a reduction in accuracy, resulting in convergence behavior that becomes closer to that of the Trefftz approaches.
This might be explained by the fact that the quasi-Trefftz and embedded Trefftz spaces are constructed so that the elemental residual $\calL u_h-f$ is small. 
In general, the numerical results indicate that the least-squares variant does not provide an improvement.

\begin{figure}[!htbp]  
	\includegraphics[width=\linewidth]{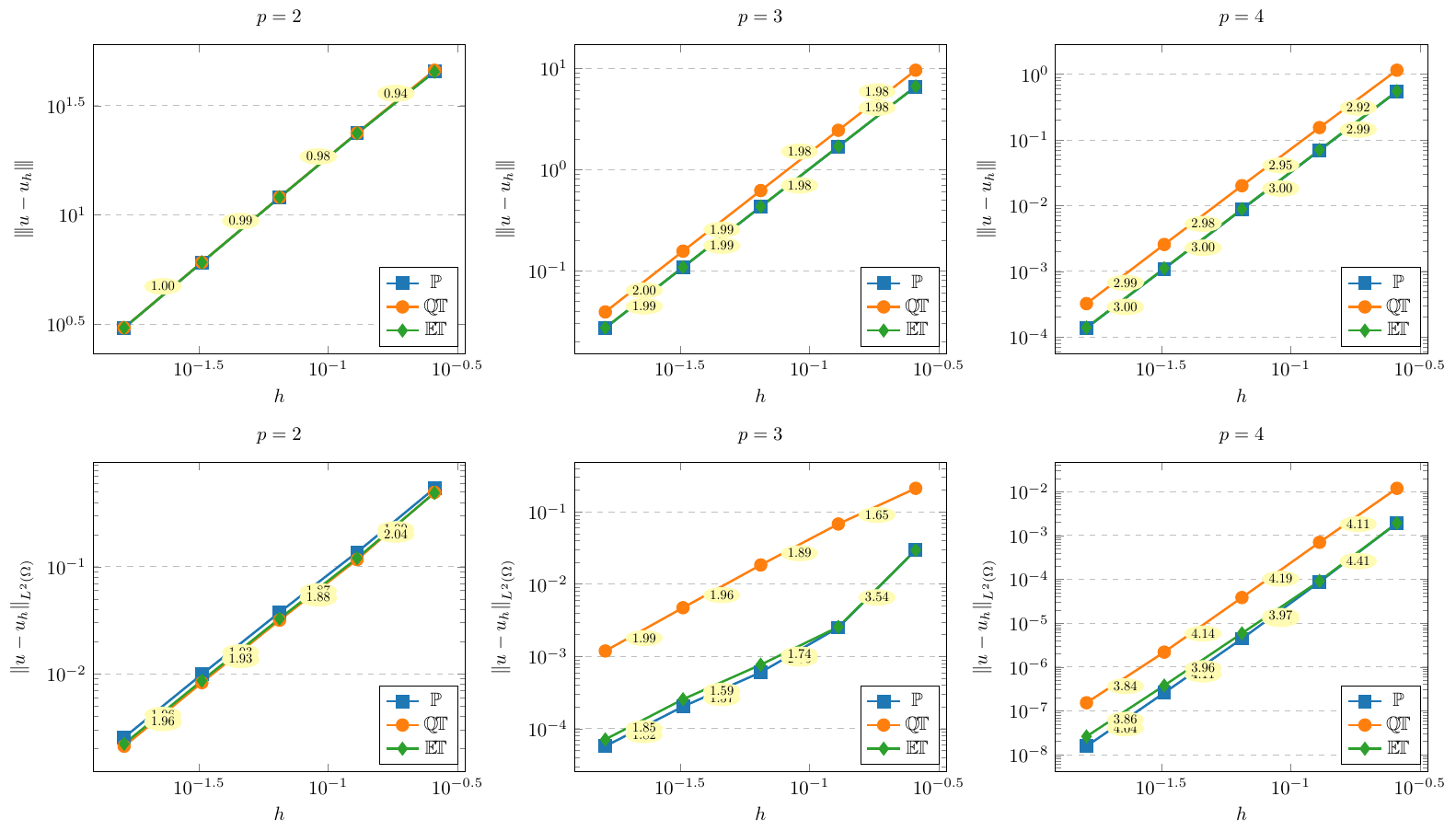}
	\caption{$h$-convergence in the norms $\vertiii{\cdot}$ (first row)  and $\Norm{\cdot}_{L^2(\Omega)}$ (second row) for the Tricomi problem with exact solution $u$ in~\eqref{eq:usolExample1} for the three discrete spaces considered using the least-squares variant of the method with $\gamma_4=1$.
	The empirical algebraic convergence rates are shown along each segment.}
	\label{fig:hconv4}
\end{figure}

We also compute the $L^2$ error separately over the elliptic region $\Omega_E := \Omega \cap \{y > 0\}$ and the hyperbolic region $\Omega_H := \Omega \cap \{y < 0\}$, in order to analyze how the method behaves in each part of the domain. 
The convergence rates in $\Omega_E$ and $\Omega_H$ are comparable.
The value of the $L^2(\Omega_E)$ norm of the error is considerably smaller than the $L^2(\Omega_H)$ error norm, reflecting the ratio between $\|u\|_{L^2(\Omega_E)}$ and $\|u\|_{L^2(\Omega_H)}$ for the solution $u$ in~\eqref{eq:usolExample1}.
For more details and plots see~\cite{perinati2026phdthesis}.

\subsection{\texorpdfstring{$p$}{p}-convergence}\label{s:pconv}

We study the $p$-convergence of the proposed method by increasing the polynomial degree $p$ on a fixed mesh.
In Figure~\ref{fig:pconv} we compare the errors in both energy and $L^2$ norms for the standard, quasi-Trefftz and embedded Trefftz DG methods on a mesh with $h=0.2$ and for polynomial degrees $p = 2,\dots,8$. 
The quasi-Trefftz and embedded Trefftz versions of the method achieve higher accuracy than the standard DG method for comparable numbers of degrees of freedom, denoted  $\mathrm{N}_{\mathrm{dofs}}$, especially 
for higher polynomial degree $p$.
We observe that the error decays exponentially with order $\mathcal{O}(e^{-A\mathrm{N}_{\mathrm{dofs}}})$ for the quasi-Trefftz and embedded Trefftz polynomial spaces, and only with root-exponential order $\mathcal{O}(e^{-B\sqrt{\mathrm{N}_{\mathrm{dofs}}}})$ for the standard polynomial space.  
In the $L^2(\Omega)$-norm, the quasi-Trefftz space shows a greater improvement when increasing the polynomial degree from an odd to the next even degree, than from an even to the next odd degree, consistent with what has been observed in the $h$-convergence results in Figure~\ref{fig:hconv}.

\begin{figure}[!htb]
	\includegraphics[width=\linewidth]{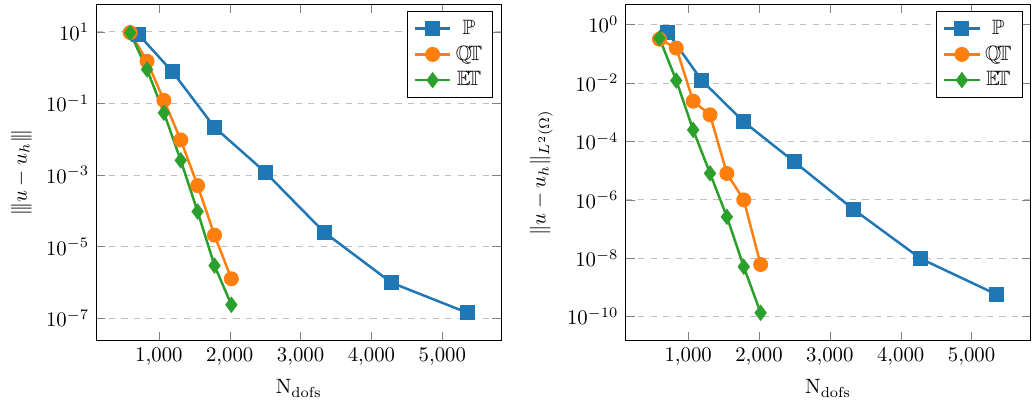}
	 \caption{$p$-convergence in the norms $\vertiii{\cdot}$ (left)  and $\Norm{\cdot}_{L^2(\Omega)}$ (right) for the Tricomi problem with exact solution $u$ in~\eqref{eq:usolExample1} for the standard, quasi-Trefftz and embedded Trefftz polynomials spaces.}
	 \label{fig:pconv}
 \end{figure}
	 
\subsection{Sensitivity to the penalty parameters}\label{s:choicepenalty}

The theoretical analysis guarantees well-posedness and stability of the variational problem~\eqref{eq:variational} under the assumptions $\gamma_1>0$ and $\gamma_2,\gamma_3\ge\gamma_*$, where $\gamma_*$ is defined in~\eqref{eq:BetaGammaStar}. 
We are interested in studying the sensitivity of the numerical solution with respect to the choice of these penalty parameters. 
We consider the Tricomi problem with exact solution~$u$ given in~\eqref{eq:usolExample1}, using a mesh size $h=0.1$ and polynomial degrees $p=2, 3, 4$.
For simplicity, we set $\gamma_2=\gamma_3$ and vary both $\gamma_1$ and $\gamma_2=\gamma_3$ in the set
\begin{align}\label{eq:setpenalty}
	&\{10^{s}: s \text{ are } 30 \text{ equispaced nodes in } [-5,5]\}.
\end{align}
For each choice of the penalty parameters, we compute the $L^2(\Omega)$ error of the numerical solution. The results are reported in Figure~\ref{fig:varypenalty}.
The numerical experiments indicate that taking both penalty parameters too small leads to large errors, indicating a loss of stability. 
If at least one of the two penalty parameters is sufficiently large than the method is stable. 
We also observed that excessively large values of $\gamma_2 = \gamma_3$ may lead to a slight loss of accuracy.
These observations suggest that the theoretical condition of $\gamma_2$ and $\gamma_3$ being sufficiently large is sufficient but not necessary to guarantee stability.
Overall, the results indicate that, while an optimized 
choice of penalties can improve accuracy, the method is robust and stable over a wide range of values of the parameters.
In particular, the embedded Trefftz method seems the most robust compared to the others, and even polynomial degrees generally behave better than odd ones. 
\begin{figure}[!htb]
	\begin{center}
	\includegraphics[width=\linewidth]{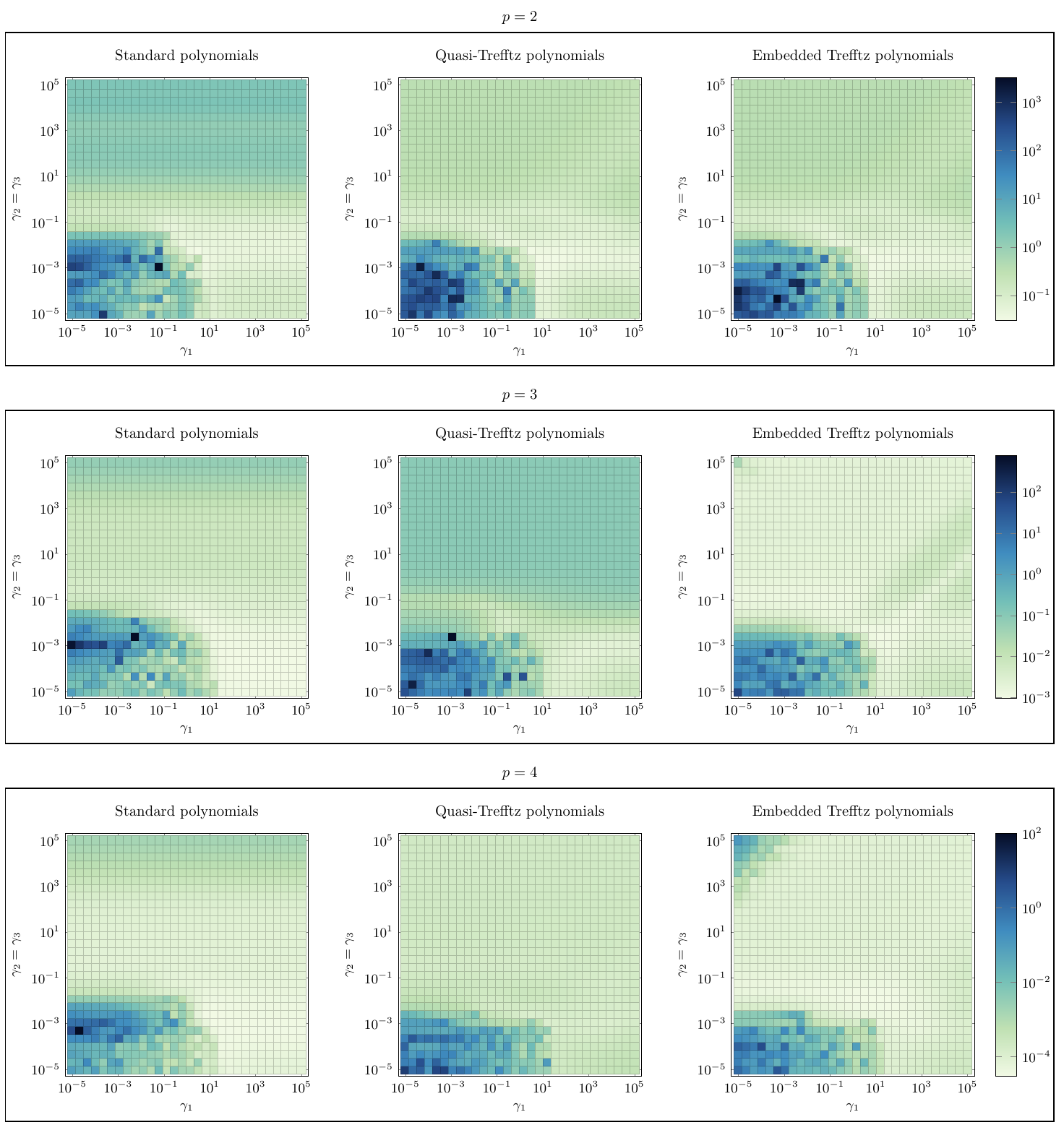}
	\end{center}
	\caption{
	$L^2(\Omega)$ error of the numerical solution with $\gamma_1$ and $\gamma_2=\gamma_3$ varying in the set~\eqref{eq:setpenalty}, for mesh size $h = 0.1$ and for the standard, quasi-Trefftz, and embedded Trefftz polynomial spaces. Results are shown for $p = 2$ (top row), $p = 3$ (middle row) and $p=4$ (bottom row).
	}
	\label{fig:varypenalty}
\end{figure}
\section{Conclusions}
\label{s:conclusions}
	We have introduced a discontinuous Galerkin formulation for the numerical discretization of a class of elliptic-hyperbolic problems, based on  the Morawetz multiplier technique.
	Coercivity is established in an energy norm, while continuity is proved in a stronger residual norm.
	We derived \emph{a priori} error estimates in the energy norm and established $hp$-error bounds for standard polynomial spaces and $h$-error bounds for quasi-Trefftz polynomial spaces.
	The numerical results for the Tricomi equation exhibit convergence rates of order at least $O(h^{p-1})$ in the energy norm, as expected from the theory, for all the discrete spaces considered: standard, quasi-Trefftz and embedded Trefftz polynomial spaces. 
	The quasi-Trefftz and embedded Trefftz approaches, in particular, achieve comparable accuracy with a significant reduction of the number of degrees of freedom. 
	The influence of the penalty parameters on the performance of the method has also been investigated.

	A possible direction for future work is the extension of the analysis to more general elliptic--hyperbolic operators.
	The inclusion of lower-order terms could be treated using techniques similar to those developed here, following, for example, \cite{aziz1984finite} for the treatment of a zeroth-order term.
	A more challenging extension concerns operators with non-diagonal principal parts or coefficients $K(x,y)$ depending on both spatial coordinates. 
	In particular, Keldysh-type operators (such as $K(x) u_{xx} + u_{yy}$) are generally more difficult to analyze than the Tricomi-type operators $K(y)u_{xx}+u_{yy}$ considered in this work (see \cite[Intro., \S 3.1--3.2, App. A.5]{otway2012dirichlet} and \cite[\S 2.5]{otway2015elliptic}).	
\begin{table}[htb!]
	\centering
	\begin{tblr}{
			colspec = {|c |c |c|},
			row{1} = {font=\bfseries},
			width = \linewidth
		}
		\hline
		\textbf{Symbol} & \textbf{Meaning} & \textbf{Definition}\\
		\hline
		$\calL$ & Frankl operator $\calL u:= Ku_{xx}+u_{yy}$&\eqref{eq:PDE}\\
		$K=K(y)$ & Sign-changing PDE coefficient ($K(y)=y$ Tricomi)& \eqref{eq:kappa}\\
		$\Omega$ & Computational domain &\S\ref{s:introduction}\\
		$\Gamma_0,\Gamma_1,\Gamma_2$ & Elliptic boundary, left and right characteristic &\S\ref{s:introduction}\\
		$\calT_h,\calT_\calH,\calH$ & Mesh, mesh sequence, mesh size sequence & \S\ref{s:mesh} \\
		$\calF_h,\FhI,\FhD,\calF_h^j,\calF_T$ & Facet sets, $j=0,1,2$ & \S\ref{s:mesh}\\
		$h,h_T,h_F$ & Mesh size, element and facet diameters & \S\ref{s:mesh}\\
		$H^m\Th,\IP^p\Th$ & Broken (elementwise) Sobolev and polynomial spaces & \S\ref{s:mesh}\\
		$\mvl\cdot,\jmp{\cdot}$ & Average and jump operators & \S\ref{s:mesh}\\
		$r_\star,C_g$ & Star-shaped and graded-mesh parameters& \S\ref{s:mesh} \ref{it:starshaped}--\ref{it:graded}\\
		$\Ctr$ & Inverse trace inequality constant & 
		\eqref{eq:discretetraceinequality}\\			
		$V_h,V_*,V_{*h}$ & Discrete and continuous function spaces &\S\ref{s:DGformulation}\\
		$\bm=(b, c)^\top$ & Morawetz coefficient vector and functions & \eqref{eq:Multiplier} \\
		$\calM$ &  Morawetz multiplier $\calM v:=bv_x+c v_y$ & \eqref{eq:Multiplier}\\
		$\W$ & Frankl operator matrix $\W:=\begin{bmatrix}K&0\\0 & 1\end{bmatrix}$&\S\ref{s:DGformulation}\\
		$\delta$ & Morawetz multiplier positivity parameter & \eqref{eq:delta}\\
		$\gamma_1,\gamma_2,\gamma_3,\gamma_4 $& Penalty coefficients for $\jmp{u}$, $\jmp{u_x}$, $\jmp{u_y}$, $u_t$ and $\calL u$ &\eqref{eq:penalty}, Rem.~\ref{rem:LS} \\
		$\calA_h,\calA_J,L_h$ & DG bilinear and linear forms &\eqref{eq:Ah},\eqref{eq:AJ} \\
		$\vertiii{\cdot},\ \abs{\cdot}_J,\ \vertiii{\cdot}_{\calL}$ & 
		Energy norm, jump seminorm, residual norm & \eqref{eq:energynorm}, \eqref{eq:jumpseminorm},  \eqref{eq:residualnorm}\\		
		$\int_{\calT_h},\int_{\partial \calT_h},\int_{\calF_h}$ & Elementwise and facetwise integrals &\S\ref{s:Norms}\\
		$Q_n,Q_t,Q_{nt},\mathbf M$ & Boundary integrand normal and tangential parts & Lemma~\ref{lem:Qdecomposition}, \eqref{eq:QnQtQnt}\\
		$\beta,\gamma_*$ & Coefficient size, penalty threshold & Prop.~\ref{prop:coercivity}, \eqref{eq:BetaGammaStar}\\
		$\jmp{\cdot}_x,\jmp{\cdot}_y$ & Cartesian components of normal jump & \eqref{eq:jumpxy}\\
		$M$ & DG bilinear form continuity constant & Prop.~\ref{prop:continuity}\\
		$\Tc,\mathfrak T,N_\Omega,h_{\mathfrak{T}}$ & Mesh covering: mesh, elements, parameters & Ass.~\ref{ass:covering}\\
		$\mathfrak{E}_{D}$ & Stein's extension operator & \eqref{eq:Stein}\\
		$\Capp,\Pi_{h,p}$ & Polynomial approximation constant and projector & \eqref{eq:Pihp_T}\\
		$\qT^p_f(T),\qT^p_f(\calT_h)$ & Local and global quasi-Trefftz spaces & \S\ref{s:quasiTrefftz}\\
		$\eT^p_f(T),\eT^p_f(\calT_h)$ & Local and global embedded Trefftz spaces & \S\ref{s:embeddedTrefftz}\\
		$d$ & Elliptic domain height & \eqref{eq:Gamma0experiments}\\
		$b_0,b_1,c_0,c_1$ & Affine Morawetz multiplier parameters & \eqref{eq:bcAffine}\\
		\hline
	\end{tblr}
	\caption{List of the main symbols used.}
	\label{tab:notation}
\end{table}

\section*{Acknowledgements}
LMIG, AM and PS gratefully acknowledge the Centro Internazionale per la Ricerca Matematica (CIRM, Trento) for hosting them in the Research-in-Pairs program.
AM and CP acknowledge support from the PRIN project ``ASTICE'' (202292JW3F) funded by the European Union -- NextGenerationEU, and from GNCS--INDAM.
This research was funded in part by the Austrian Science Fund (FWF) \href{https://doi.org/10.55776/ESP4389824}{10.55776/ESP4389824}.
For open access purposes, the authors have applied a CC BY public copyright license to any author-accepted manuscript version arising from this submission.
LMIG acknowledges support from the US National Science Foundation (NSF): this material is based upon work supported by the NSF under Grant No. \href{https://www.nsf.gov/awardsearch/showAward?AWD_ID=2110407&HistoricalAwards=false}{DMS-2110407}.
LMIG has disclosed an outside interest in Airbus Central R\&T to the University of Arizona. Conflicts of interest
resulting from this interest are being managed by The University of Arizona in accordance with its policies.

\bibliographystyle{plain}
\bibliography{bibliography.bib}	
 
\end{document}